\documentclass[10pt]{article}

\usepackage{amsmath,amssymb,amsfonts,amsbsy,mathrsfs}

\renewcommand{\subsection}{\subsubsection}

\newtheorem{theorem}{Theorem}[section]

\newtheorem{lemma}{Lemma}[section]
\newtheorem{proposition}{Proposition}[section]
\newenvironment{proof}{\noindent{\bf Proof}.}{\hfill $\Box$ \vspace*{4mm}}
\newtheorem{definition}{Definition}[section]

\newtheorem{remark}{Remark}[section]

\def\nt{|\hspace{-0.7pt}|\hspace{-0.7pt}|}
\def\nl{\langle\hspace*{-2pt}\langle}
\def\nr{\rangle\hspace*{-2pt}\rangle}

\begin{document}

\begin{center}
{\LARGE \bf Local existence for the free boundary problem }\\[6pt]
{\LARGE \bf for the non-relativistic and relativistic compressible}\\[6pt]
{\LARGE \bf Euler equations with a vacuum boundary condition}
\end{center}

\vspace*{7mm}

\centerline{\large Yuri Trakhinin}

\begin{center}
Sobolev Institute of Mathematics, Koptyug av. 4, 630090 Novosibirsk, Russia\\
e-mail: trakhin@math.nsc.ru
\end{center}

\vspace*{7mm}
\centerline{\bf \large Abstract}

\noindent We study the free boundary problem for the equations of compressible Euler equations with a vacuum boundary condition. Our main goal is to recover in Eulerian coordinates the earlier well-posedness result obtained by Lindblad \cite{Lind} for the isentropic Euler equations and extend it to the case of full gas dynamics. For technical simplicity we consider the case of an unbounded domain whose boundary has the form of a graph and make short comments about the case of a bounded domain. We prove the local-in-time existence in Sobolev spaces by the technique applied earlier to weakly stable shock waves and characteristic discontinuities \cite{CS,Tr}. It contains, in particular, the reduction to a fixed domain, using the ``good unknown'' of Alinhac \cite{Al}, and a suitable Nash-Moser-type iteration scheme. A certain modification of such an approach is caused by the fact that the symbol associated to the free surface is not elliptic. This approach is still directly applicable to the relativistic version of our problem in the setting of special relativity and we briefly discuss its extension to general relativity.

\section{Introduction}
\label{s1}

Consider the compressible Euler equations with the gravitational field ${\cal G}\in \mathbb{R}^3$:
\begin{align}
& \partial_t\rho  +{\rm div}\, (\rho v )=0,  \label{1}\\[3pt]
& \partial_t(\rho v ) +{\rm div}\,(\rho v\otimes v  ) +
{\nabla}p=\rho {\cal G}, \label{2} \\[3pt]
& \textstyle \partial_t\left( \rho \left(e+  \frac{1}{2}|v|^2 \right)\right)+
{\rm div}\,  \left( \left(\rho  \left(e+\frac{1}{2}|v|^2\right) +p\right)v \right)=0,  \label{3}
\end{align}
where $\rho$ denotes density, $v\in\mathbb{R}^3$ fluid velocity, $p=p(\rho,S )$ pressure, $S$ entropy, and  $e=e(\rho,S )$ internal energy. With a state equation of gas, $p=p(\rho ,S)$, and the first principle of thermodynamics, \eqref{1}--\eqref{3} is a closed system. As the unknown we can fix, for example, the vector $ U =U (t, x )=(p, v,S)$.

We can easily symmetrize system \eqref{1}--\eqref{3} by rewriting it in the nonconservative form
\begin{equation}
\frac{1}{\rho c^2}\,\frac{{\rm d} p}{{\rm d} t} +{\rm div}\, v =0,\qquad
\rho\, \frac{{\rm d} v}{{\rm d} t} +{\nabla} p  =\rho {\cal G} ,\qquad \frac{{\rm d} S}{{\rm d} t} =0,
\label{4}
\end{equation}
where $c^2=p_\rho(\rho ,S)$ is the square of the sound velocity and ${\rm d} /{\rm d} t =\partial_t+({v} ,{\nabla} )$ (by $(\ ,\ )$ we denote the scalar product). Equations (\ref{4}) read as the symmetric quasilinear system
\begin{equation}
\label{5}
A_0(U )\partial_tU+\sum_{j=1}^3A_j(U )\partial_jU+Q(U)=0,
\end{equation}
where $Q(U)=(0,-\rho{\cal G},0)$,
\[
A_0=\left( \begin{array}{ccccc} {\displaystyle \frac{1}{\rho c^2}} & 0 & 0 & 0 & 0 \\[6pt] 0 & \rho & 0 & 0 & 0\\
0 & 0 & \rho & 0 & 0 \\
0 & 0 & 0 & \rho & 0\\
0 & 0 & 0 & 0 & 1 \end{array} \right) ,\quad
A_1=\left( \begin{array}{ccccc} {\displaystyle \frac{v_1}{\rho c^2}}&1&0&0&0\\[6pt]
 1&\rho v_1&0&0&0\\
0&0&\rho v_1&0&0\\ 0&0&0&\rho v_1&0\\0&0&0&0&v_1
\end{array} \right) ,
\]
\[
A_2=\left( \begin{array}{ccccc} {\displaystyle \frac{v_2}{\rho c^2}}&0&1&0&0\\[6pt]
 0&\rho v_2&0&0&0\\ 1&0&\rho v_2&0&0\\ 0&0&0&\rho v_2 &0\\ 0&0&0&0&v_2
\end{array} \right) ,\quad
A_3=\left( \begin{array}{ccccc} {\displaystyle \frac{v_3}{\rho c^2}}&0&0&1&0\\[6pt]
 0&\rho v_3&0&0&0\\ 0&0&\rho v_3&0&0\\ 1&0&0&\rho v_3&0 \\ 0&0&0&0&v_3
\end{array} \right) .
\]
System (\ref{5}) is symmetric hyperbolic if the the hyperbolicity condition $A_0>0$ holds:
\begin{equation}
\rho  >0,\quad p_\rho >0. \label{6}
\end{equation}
One can alternatively consider the {\it isentropic} Euler equations, i.e., system (\ref{1}), (\ref{2}) for the same variables except for the entropy $S$.  Then, the state equation of gas is $p=p(\rho )$ and the second inequality in
\eqref{6} is understood in the sense that $p\,'(\rho )>0$.

We are interested in the motion of an ideal compressible fluid (gas) body in vacuum described by the Euler equations \eqref{1}--\eqref{3} (or \eqref{1}, \eqref{2} for isentropic gas) in a space-time domain $\Omega (t)$ which boundary $\Sigma (t)=\{F(t,x)=0\}$ is to be determined and moves with the velocity of the gas particles at the boundary:
\begin{equation}
\frac{{\rm d}F }{{\rm d} t}=0,\qquad p=0 \qquad \mbox{on}\quad \Sigma (t)\label{8}
\end{equation}
(for all $t\in [0,T]$). This free boundary problem can be used for modeling the motion of the ocean or a star.
Most results for such kind of problems were earlier obtained for incompressible fluids and the history of mathematical studies of incompressible versions of problem \eqref{1}--\eqref{3}, \eqref{8} can be found, for example, in \cite{Lind}.

The first result for compressible fluids was obtained by Makino \cite{Makino1} (see also \cite{Makino}) who proved the local-in-time existence of solutions to problem \eqref{1}--\eqref{3}, \eqref{8} for the case of a polytropic gas and when the boundary condition $p=0$ in \eqref{8} is replaced by $\rho =0$. This was done by
using a special symmetrization of the gas dynamics system that supports vacuum regions. That is, the corresponding symmetric system for a new unknown $U$ (see \cite{Makino1,Makino}) is always hyperbolic without assumptions \eqref{6}.
However, employing this symmetrization leads to certain non-physical restrictions on the initial data. Therefore,
Makino's result does not cover the general case. On the other hand, from the physical point of view, the vacuum boundary condition $\rho |_{\Sigma} =0$ is, of course, more natural than $p|_{\Sigma}=0$. In particular, \eqref{6} and \eqref{8} does not formally allow the equation of state of a polytropic gas $p=a\rho^{\gamma}\exp (S/c_V)$. In this connection, as was recommended in \cite{Lind}, for the case of boundary condition $p|_{\Sigma}=0$ one can alternatively think of the pressure as a small constant on the boundary (see also Remark \ref{r1} below).

The local-in-time existence for the general case of initial data was recently proved by Lindblad \cite{Lind} for the free boundary problem with non-vanishing density on the boundary for the isentropic Euler equations.
Namely, the local-in-time existence of smooth solutions of problem \eqref{1}, \eqref{2}, \eqref{8} (with ${\cal G}=0$) was shown in \cite{Lind} under the natural physical assumption
\begin{equation}
\frac{\partial p}{\partial N}\leq -\epsilon <0\qquad \mbox{on}\quad \Sigma (0),
\label{9}
\end{equation}
where $\partial /\partial N=(\nabla F,\nabla )$, together with the hyperbolicity condition \eqref{6}, provided that the initial domain $\Omega (0)$ is diffeomorfic to a ball.
The main tool in \cite{Lind} is the passage to the Lagrangian coordinates for reducing the original problem to that in a fixed domain. Such a technique seems most natural for free boundary problems with boundary conditions like \eqref{8}. At the same time, for {\it compressible} fluids it is connected with a lot of technical difficulties and it is not quite clear how to extend the results to similar problems for more complicated fluid dynamics models like, for example,
relativistic gas dynamics or magnetohydrodynamics. Even the extension of the existence theorem in \cite{Lind} to full gas dynamics does not seem to be just a technical matter.

\begin{remark}{\rm
If the domain $\Omega (t)$ is unbounded, we should additionally assume that the velocity vanishes at infinity (as $|x|\rightarrow\infty$). As follows from the second vector equation in \eqref{4}, in the absence of gravity (${\cal G}=0$) this contradicts condition \eqref{9}. That is, in the case of an unbounded domain, the presence of gravity is {\it absolutely necessary}. However, if the domain is bounded, without loss of generality and as was done in \cite{Lind}, the gravity can be neglected as a lower order term (it plays no role in the proof of well-posedness).
}\label{r.unb}
\end{remark}

In this paper we propose another approach to studying the well-posedness of problem \eqref{1}--\eqref{3}, \eqref{8}
(or \eqref{1}, \eqref{2}, \eqref{8}) and similar free boundary problems for other systems of hyperbolic conservation laws. This approach could be probably called ``hyperbolic'' or ''shock waves'' approach because it was first applied by Blokhin (see \cite{BThand} and references therein) and Majda \cite{Majda} to prove the short-time persistence of discontinuous shock front solutions to hyperbolic conservation laws.  The ``hyperbolic'' approach to free boundary problems does not propose to pass to the Lagrangian coordinates (the more so as this is impossible for shock waves).
Instead of this we work in the Eulerian coordinates and reduce our free boundary problem to that in a fixed domain.
More precisely, such a procedure is indeed quite simple if our domain $\Omega (t)$ is unbounded and its boundary has the form of a graph. In this case we reduce our problem to that in a half-space by simple straightening of the unknown free surface (for example, a shock front). Otherwise, the technique of reduction to a fixed domain is more technically involved (see \cite{Majda}), but the resulting problem in a fixed domain has no principal differences from that for the case of unbounded domains. We can then follow standard arguments and reduce the corresponding linearized problem to a linear problem in a half-space by using a fixed partition of unity flattering the boundary. Therefore, without loss of generality we can restrict ourself to an unbounded initial domain and we do so in this paper. On the other hand, the possibility to treat unbounded domains is already a certain advantage of the ``hyperbolic'' approach.

Regarding the free boundary problem \eqref{1}--\eqref{3}, \eqref{8}, it should be noted that its linearized version is well-posed only in a weak sense. It means that the corresponding linear problem satisfies the Kreiss--Lopatinski condition but violates the uniform Kreiss--Lopatinski condition \cite{Kreiss,Majda,Met}. This yields losses of derivatives in a priori estimates for the linearized problem. Therefore, we are not able to use such estimates to prove the existence of solutions to the original nonlinear problem by the fixed-point argument as was done by Blokhin or Majda (see also \cite{Met}) for uniformly stable shock waves (the uniform Kreiss--Lopatinski condition holds for such shocks). Thus, we have to modify the ``hyperbolic'' approach to apply it to free boundary problems whose linearized versions are weakly well-posed. In some sense, this was already done in previous works. We should first mention Alinhac's study \cite{Al} of rarefaction waves for hyperbolic conservation laws. 

It is well-known that the Nash-Moser method can sometimes compensate the loss of derivatives phenomenon and to use it we should perform a genuine linearization of our nonlinear problem, i.e., to keep all the lower-order terms while linearizing. One of these terms is a first-order term for the perturbation of the free surface in the linearized interior equations. To neutralize such a bad term Alinhac proposed to pass to a new unkwnown (so-called ``good unknown'') and we use this idea for problem \eqref{1}--\eqref{3}, \eqref{8}.
Such a technique was recently applied to other hyperbolic free boundary value problems. We mean the results of Coulombel and Secchi \cite{CS}  for 2D supersonic vortex sheets and weakly stable shock waves in isentropic gas dynamics and author's result for compressible current-vortex sheets \cite{Tr1,Tr}. The local-in-time existence of the listed
weakly stable discontinuities was shown in \cite{CS,Tr} by a suitable Nash-Moser-type iteration scheme.

At last, we should note that problem \eqref{1}--\eqref{3}, \eqref{8} is not a quite standard ``weakly stable'' hyperbolic
free boundary problem like those studied in \cite{Al,CS,Tr}. Actually, regardless of the fact that the constant (``frozen'') coefficients linearized problem for \eqref{1}--\eqref{3}, \eqref{8} always satisfies the weak Kreiss--Lopatinski condition, the corresponding variable coefficients problem is not unconditionally well-posed and 
\eqref{9} is an {\it extra} condition which is necessary for well-posedness (though, the question on its necessity is a separate and non-trivial problem). This unusual feature is a consequence of the fact that the symbol associated with the free surface is {\it not elliptic} (see Remark \ref{r2}) that leads to a loss of ``control on the boundary.''  Therefore, we have to modify somewhat the energy method which we use for deriving a priori estimates for the linearized problem. Having in hand a good a priori estimate (so-called {\it tame estimate} \cite{Al}) for the linearized problem, we prove the local existence (and uniqueness) theorem for our nonlinear problem (see Theorem \ref{t2} below) by the Nash-Moser method.

Such a modified ``hyperbolic'' approach outlined above allows one to prove a counterpart of Theorem \ref{t2} for the relativistic version of problem \eqref{1}--\eqref{3}, \eqref{8} in the setting of special relativity without further modifications. Actually, the proof is absolutely the same as for the non-relativistic case and we may drop it. Since in the framework of our ``hyperbolic'' approach we use the energy method (but not the Kreiss symmetrizer technique \cite{Kreiss,Majda,Met}), the only important point is that the system of relativistic Euler equations 
\begin{equation}
\nabla_{\alpha}(\rho u^{\alpha})=0,\qquad
\nabla_{\alpha}T^{\alpha\beta}=0
\label{GR}
\end{equation}
can be symmetrized (we write down its symmetric form in the last section of the paper). Here $\nabla_{\alpha}$ is the covariant derivative with respect to the metric $g$ with the components $g_{\alpha\beta}$; $\rho$ is the particle number density in the rest frame (for convenience we use the notations that are consistent with the non-relativistic case);
\[
T^{\alpha\beta}=\rho h u^{\alpha}u^{\beta}+ p g^{\alpha\beta};
\]
$h= 1 + e + (p/\rho )$ is the specific enthalpy, $p$ is the pressure, $e= e(\rho ,S)$ is the specific internal energy per particle, $S$ is the entropy per particle, $u^{\alpha}$ are components of the four-velocity. The metric $g$ should satisfy the Einstein equations. Following \cite{Rend} (see also \cite{FR}), in the last section of the paper we write down them in so-called harmonic coordinates. In the case of special relativity $g={\rm diag}\,(-1,1,1,1)$ and equations \eqref{GR} (in the presence of gravity) take the form 
\begin{align}
& \partial_t(\rho\Gamma ) +{\rm div}\, (\rho u )=0,  \label{SR1}\\[3pt]
& \partial_t(\rho h \Gamma u ) +{\rm div}\,(\rho h u\otimes u  ) +
{\nabla}p=\rho{\cal G} , \label{SR2} \\[3pt]
& \partial_t(\rho h\Gamma^2-p ) +{\rm div}\, (\rho h\Gamma u )=0,  \label{SR3}
\end{align}
where 
\[
t:=x^0,\quad {\rm div} :={\rm div}_x,\quad x=(x^1,x^2,x^3),\quad u=(u^1,u^2,u^3),\quad v=(v^1,v^2,v^3)=u/\Gamma,\quad
\Gamma^2 =1+|u|^2;
\]
$\Gamma =u^0=(1-|v|^2)^{-1/2}$ is the Lorentz factor, and the speed of the light is equal to unity.

Regarding the free boundary problem for relativistic fluids with a vacuum boundary condition, its local-in-time existence was proved by Rendall \cite{Rend} for the boundary condition $\rho|_{\Sigma} =0$ and a special class of initial data by generalizing Makino's symmetrization \cite{Makino1,Makino} to the relativistic case. This result was obtained for the setting of general relativity and under the simplifying assumption that the relativistic fluid
is isentropic. Actually, in the framework of Makino's approach this assumption was just a technical simplification. That is, our main goal in this paper is to cover the general case of initial data but for the boundary condition $p|_{\Sigma} =0$. 

As was already noted above, we do not almost need to make efforts for extending Theorem \ref{t2} to the relativistic Euler equations in the setting of special relativity. Concerning the case of general relativity, the proof of the existence theorem is based on using harmonic coordinates and the facts that the Einstein equations for the metric $g$ can be written in the form of a symmetric hyperbolic system \cite{Rend} and the metric should be smooth on the fluid-vacuum boundary $\Sigma$. More precisely, for the relativistic Euler equations we easily obtain a counterpart of Theorem \ref{t2} for any fixed metric, but not only for $g={\rm diag}\,(-1,1,1,1)$.  Then, roughly speaking, we resolve the relativistic Euler equations by Nash-Moser iterations whereas at each Nash-Moser iteration step we find the metric from the Einstein equations by Picard iterations. Actually, we do not even need to write down Picard iterations because we
know that a unique solution to the Einstein equations (for fixed fluid unknowns) written in the form of a symmetric hyperbolic system does exist and this is proved by the classical fixed-point argument. Since it makes probably sense to devote a separate paper to the case of general relativity we restrict ourself to a schematic proof of the existence theorem. Moreover, we do not even formally write down such a theorem in this paper.

The plan of the rest of the paper is the following. In Section \ref{s3}, we reduce problem \eqref{1}--\eqref{3}, \eqref{8} to that in a fixed domain and state the existence Theorem \ref{t2} for the reduced problem. In Section \ref{s3} we also formulate the linearized problem and prove its well-posedness under suitable assumptions on the basic state about which we linearize our nonlinear problem \eqref{1}--\eqref{3}, \eqref{8}. The main of these assumptions is the physical condition (\ref{9}). In Section \ref{s4}, for the linearized problem we derive an a priori tame estimate in the Sobolev spaces $H^s$ with $s\geq 3$. In Section \ref{s5}, we first specify compatibility conditions for the initial data and, by constructing an approximate solution, reduce our problem to that with zero initial data. Then, we solve the reduced problem by a suitable Nash-Moser-type iteration scheme. At last, in Section \ref{s6} we describe extensions of the result of Theorem \ref{t2} to special and general relativity.

\section{Basic a priori estimate for the linearized problem}
\label{s3}

For technical simplicity (see Remark \ref{r1'} below), we assume that the space-time domain $\Omega (t)$ is unbounded and lies from one side of its free boundary $\Sigma (t)$ which has the form of a graph,
$x_1=\varphi (t,x')$, $x'=(x_2,x_3)$. That is,  
\begin{equation}
\Omega (t)=\left\{x_1>\varphi (t,x')\right\}\label{10}
\end{equation}
and the function $\varphi (t,x')$ is to be determined.
As for shock waves, using Majda's arguments \cite{Majda} , we can generalize the technique below to the case of an arbitrary compact free surface $\Sigma$. The mapping of $\Omega (t)$ to a fixed domain is just more technically involved when $\Omega (t)$ is bounded (see Remark \ref{r1'}).

For domain \eqref{10} the boundary conditions \eqref{8} take the form
\begin{equation}
\partial_t\varphi =v_N,\qquad p=0 \qquad \mbox{on}\quad \Sigma (t),\label{11}
\end{equation}
and the gravitational field
\[
{\cal G}= ( G , 0 , 0),
\]
where $v_N=(v,N)$, $N=(1,-\partial_2\varphi ,-\partial_2\varphi )$, and $G$ denotes Newton's gravitational constant.
Our final goal is to find conditions on the initial data
\begin{equation}
U (0,x)=U_0(x),\quad x\in \Omega (0),\qquad \varphi (0,x')=\varphi_0(x'),\quad x'\in\mathbb{R}^2,\label{12}
\end{equation}
providing the existence of a smooth solution $(U ,\varphi )$ of the free boundary value problem (\ref{5}),
(\ref{11}), (\ref{12}) in $\Omega (t)$ for all $t\in [0,T]$, where the time $T$ is small enough.

To reduce the free boundary value problem (\ref{5}), (\ref{11}), (\ref{12}) to that in a fixed domain we straighten, as usual, the unknown free surface $\Sigma$. That is, the unknown $U$ being smooth in $\Omega (t)$ is replaced by the vector-function
\[
\widetilde{U}(t,x ):= U (t,\Phi (t, x),x'),
\]
that is smooth in the fixed domain $\mathbb{R}^3_+=\{x_1>0,\ x'\in \mathbb{R}^2\}$ , where
$\Phi (t,0,x')=\varphi (t,x')$ and $\partial_1\Phi >0$.  As in \cite{Tr}, to avoid assumptions about compact support of the initial data in the nonlinear existence theorem and work globally in $\mathbb{R}^3_+$ we use the choice of 
$\Phi (t,x )$ similar to that suggested by M\'etivier \cite{Met}:
\[
\Phi(t,x ):= x_1+\Psi(t,x ),\quad \Psi(t,x ):= \chi (x_1)\varphi (t,x'),
\]
where $\chi\in C^{\infty}_0(\mathbb{R})$ equals to 1 on $[0,1]$, and $\|\chi'\|_{L_{\infty}(\mathbb{R})}<1/2$. Then, the fulfillment of the requirement $\partial_1\Phi >0$ is guaranteed
if we consider solutions for which $\|\varphi \|_{L_{\infty}([0,T]\times\mathbb{R}^2)}\leq 1$. The last is fulfilled if,
without loss of generality, we consider the initial data satisfying $\|\varphi_0\|_{L_{\infty}(\mathbb{R}^2)}\leq 1/2$,
and the time $T$ in our existence theorem is sufficiently small.

Dropping for convenience tildes in $\widetilde{U}$, we reduce (\ref{5}), (\ref{11}), (\ref{12}) to the initial boundary value problem
\begin{equation}
\mathbb{L}(U,\Psi)=0\quad\mbox{in}\ [0,T]\times \mathbb{R}^3_+,\label{13}
\end{equation}
\begin{equation}
\mathbb{B}(U,\varphi )=0\quad\mbox{on}\ [0,T]\times\{x_1=0\}\times\mathbb{R}^{2},\label{14}
\end{equation}
\begin{equation}
U|_{t=0}=U_0\quad\mbox{in}\ \mathbb{R}^3_+,\qquad \varphi|_{t=0}=\varphi_0\quad \mbox{in}\ \mathbb{R}^{2},\label{15}
\end{equation}
where $\mathbb{L}(U,\Psi)=L(U,\Psi)U +Q(U)$,
\[
L(U,\Psi)=A_0(U)\partial_t +\widetilde{A}_1(U,\Psi)\partial_1+A_2(U )\partial_2+A_3(U )\partial_3,
\]
\[
\widetilde{A}_1(U,\Psi )=\frac{1}{\partial_1\Phi}\Bigl(
A_1(U )-A_0(U)\partial_t\Psi-\sum_{k=2}^3A_k(U)\partial_k\Psi \Bigr)
\]
($\partial_1\Phi= 1 +\partial_1\Psi$), and (\ref{14}) is the compact form of the boundary conditions
\[
\partial_t\varphi-v_N=0,\qquad p=0 \qquad\mbox{on}\ [0,T]\times\{x_1=0\}\times\mathbb{R}^{2}.
\]

We are now in a position to state the local-in-time existence theorem for problem \eqref{13}--\eqref{15}. Clearly, this theorem implies a corresponding theorem for the original problem (\ref{5}), (\ref{11}), (\ref{12}).

\begin{theorem} Let $m\in\mathbb{N}$ and $m\geq 6$. Suppose the initial data (\ref{13}), with
\[
(U_0-\check{U},\varphi_0)\in H^{m+7}(\mathbb{R}^3_+)\times H^{m+7}(\mathbb{R}^2)\quad \mbox{and}\quad
\rho(p_0,S_0)-\epsilon_1\in H^{m+7}(\mathbb{R}^3_+),
\]
satisfy the hyperbolicity condition \eqref{6} for all $x\in\overline{\mathbb{R}^3_+}$  and are compatible up to order $m+7$ in the sense of Definition \ref{d1}. Here 
\[
\check{U}= (2\epsilon x_1,0,0,0,0),\quad \epsilon_1=2\epsilon /G\quad  (\epsilon ={\rm const} >0).
\] 
Let also the initial data satisfy the physical condition 
\begin{equation}
\partial_1p \geq \epsilon >0\qquad\mbox{at}\quad x_1=0
\label{16}
\end{equation}
for all $x'\in \mathbb{R}^2$. Then, there exists a sufficiently short time $T>0$ such that problem (\ref{13})--(\ref{15}) has a unique solution
\[
(U,\varphi )\in \left\{\check{U}+H^m([0,T]\times\mathbb{R}^3_+)\right\}\times H^m ([0,T]\times\mathbb{R}^2).
\]
Moreover, $\rho -\epsilon_1\in H^{m}([0,T]\times\mathbb{R}^3_+)$.
\label{t2}
\end{theorem}

\begin{remark}{\rm 
The hyperbolicity condition \eqref{6} which should be satisfied for all $x\in\overline{\mathbb{R}^3_+}$ implies that the function $\rho_0(x)=\rho (p_0,S_0)(x)$ cannot vanish at infinity. Indeed, in Theorem \ref{t2} we assume that $\rho_0- \epsilon_1\in H ^{m+7}(\mathbb{R}^3_+)$. On the other hand, \eqref{6} together with the boundary condition $p|_{\Sigma}=0$ do not formally allow the equation of state of a polytropic gas (or a $\gamma$--law gas for isentropic gas dynamics). However, as was noted in \cite{Lind}, from a physical point of view we can alternatively think of the pressure as a small positive constant $\varepsilon$ on the boundary. One can easily generalize the result of Theorem \ref{t2} to the case of the boundary condition $p|_{x_1=0}=\varepsilon$. More precisely, we now assume that $U_0-\check{U}- C_0 \in H ^{m+7}(\mathbb{R}^3_+)$ and prove that $U -\check{U}-C_0 \in H ^{m}([0,T]\times\mathbb{R}^3_+)$, where
$C_0=(\varepsilon , 0,0,0,0)$. Indeed, making the change of unknown $p'=p-\varepsilon$ and omitting the primes, we obtain problem (\ref{13})--(\ref{15}) with the matrices $A_{\alpha}(U+C_0)$. The further arguments are almost the same as in the proof of Theorem \ref{t2} (see below).
}
\label{r1}
\end{remark}

\begin{remark}{\rm 
Inequality \eqref{16} is a counterpart of the physical condition \eqref{9} for the unbounded domain \eqref{10}. 
If the domain is bounded and its initial boundary $\Sigma (0)$  is a compact co-dimension one surface in $\mathbb{R}^3$, we can follow Majda's arguments \cite{Majda} (see also \cite[sect. 12.4.2]{BS}). More precisely, we can make ({\it locally in time}) a change of variables that sends all boundary locations $\Sigma (t)$ to the initial surface $\Sigma (0)$. We refer the reader to \cite{Majda,BS} for details of such a change of variables. In particular, it requires the application of the Weingarten map while writing down boundary conditions on $\Sigma (0)$.
The resulting initial boundary value problem is a problem in the fixed domain $\Omega (0)$.
Its principal difference from problem (\ref{13})--(\ref{15}) is that we have to deal with a problem
in a fixed compact domain instead of a half-space. For this problem the proof of a counterpart of Theorem \ref{t2} is more technical, but the ideas are basically the same as for Theorem \ref{t2}. For instance, we should reduce the corresponding linearized problem to that in a half-space by using a fixed partition of unity flattering the boundary. The resulting linearized problem in a half-space will not have principal differences from the linearized problem for (\ref{13})--(\ref{15}). Only its coefficients will be more technically complicated than those for the linearization of (\ref{13})--(\ref{15}). Therefore, as is usually done for shock waves or other types or strong discontinuities (see, e.g., \cite{BS,BThand,Met}), in this paper we restrict ourself to the case of an unbounded domain
whose boundary has a form of a graph.
}
\label{r1'}
\end{remark}

The existence of solutions in Theorem \ref{t2} will be proved by Nash-Moser iterations. 
The main tool for proving the convergence of the Nash-Moser iteration scheme is a so-called tame estimate \cite{Al,CS,Tr} for the linearized problem. In this section, we derive a basic a priori $L_2$--estimate for the linearized problem by the energy method. This estimate is a basis for deriving the tame estimate in Sobolev spaces (see the next section) and implies {\it uniqueness} of a solution to the nonlinear problem \eqref{13}--\eqref{15} that
can be proved by standard argument.

Let us first pass to the new unknown $U'=(p',v,S)=U-\check{U}$. For $U'$ system \eqref{13} is rewritten as
\[
\mathbb{L}'(U',\Psi):=L(U'+\check{U},\Psi )U' +{A}_{\nu} (U'+\check{U},\Psi)\partial_1\check{U}+Q(U'+\check{U})=0,
\]
where $\partial_1\check{U}=( 2\epsilon , 0 ,0, 0, 0)$. Let 
$\rho' (p',S):=\rho (p,S)$, $A'_{\alpha}(U'):=A_{\alpha}(U)$, $Q'(U'):=Q(U)$, and $U'_0:=U_0-\check{U}$. Then, omitting the primes, for the new unknown we get the system
\begin{equation}
\mathbb{L}(U,\Psi):=L(U,\Psi )U +{A}_{\nu} (U,\Psi)\partial_1\check{U}+Q(U)=0\quad\mbox{in}\ [0,T]\times \mathbb{R}^3_+\label{newL}
\end{equation}
with the boundary conditions \eqref{14} and the initial data \eqref{15}. From now on we will work with problem
\eqref{newL}, \eqref{14}, \eqref{15}. We should now prove the existence of its solution, $U \in H^m([0,T]\times\mathbb{R}^3_+)$, assuming that $U_0\in H^{m+7}(\mathbb{R}^3_+)$. For the initial data for the new unknown we assume that
\begin{equation}
\partial_1p|_{x_1=0}  > -\epsilon \quad \forall \ x'\in \mathbb{R}^2.
\label{16'}
\end{equation}
This guarantees the fulfillment of assumption \eqref{16} for the original unknown.

\begin{remark}{\rm 
We easily compute the boundary matrix:
\[
\widetilde{A}_1({U},{\Psi})=\frac{1}{\partial_1{\Phi}}
\left(
\begin{array}{ccccc} 
{\displaystyle \frac{{\mathfrak{f}}}{{\rho} {c}^2}}&1&-\partial_2{\Psi}&-\partial_3{\Psi}&0\\[6pt]
 1&{\rho} {\mathfrak{f}}&0&0&0\\
-\partial_2{\Psi}&0&{\rho} {\mathfrak{f}}&0&0\\ -\partial_3{\Psi}&0&0&{\rho} {\mathfrak{f}}&0\\0&0&0&0&{\mathfrak{f}}
\end{array}\right),
\]
where ${\mathfrak{f}}={v}_1-{v}_2\partial_2{\Psi}-{v}_3\partial_3{\Psi}-\partial_t{\Psi}$.
The vector-function $\widetilde{A}_1(U,{\Psi}^a)\partial_1\check{U}$ cannot belong to a Sobolev space on $\mathbb{R}^3_+$ because its second component is $2\epsilon/ (\partial_1\Phi)$. However, if problem \eqref{newL}, \eqref{14}, \eqref{15} has a solution from a Sobolev space and Theorem \ref{t2} takes place, then the sum $\widetilde{A}_1(U,{\Psi})\partial_1\check{U}+Q(U)$ 
already belongs to a Sobolev space because
\[
\frac{2\epsilon}{\partial_1\Phi}- G\rho = -G (\rho -\epsilon_1)-2\epsilon\frac{\partial_1\Psi}{\partial_1\Phi}\in
H^{m}([0,T]\times\mathbb{R}^3_+).
\]
Thus, for our case of an unbounded domain the presence of gravity is of great importance (see also Remark \ref{r.unb}).}
\label{rcheck}
\end{remark}

We now formulate the linearized problem. Consider 
\[
\Omega_T:= (-\infty, T]\times\mathbb{R}^3_+,\quad \partial\Omega_T:=(-\infty ,T]\times\{x_1=0\}\times\mathbb{R}^{2}.
\]
Let
\begin{equation}
(\widehat{U}(t,x ),\hat{\varphi}(t,{x}'))\in W^2_{\infty}(\Omega_T)\times W^2_{\infty}(\partial\Omega_T)
\label{17}
\end{equation}
be a given sufficiently smooth vector-function, with $\widehat{U}=(\hat{p},\hat{v},\widehat{S})$, and
\begin{equation}
\|\widehat{U}\|_{W^2_{\infty}(\Omega_T)}+\|\hat{\varphi}\|_{W^2_{\infty}(\partial\Omega_T)} \leq K,
\label{18}
\end{equation}
where $K>0$ is a constant. Moreover, without loss of generality we assume that $\|\hat{\varphi}\|_{L_{\infty}(\partial\Omega_T)}<1$. This implies $\partial_1\widehat{\Phi}\geq 1/2$,
with $\widehat{\Phi}(t,x ):= x_1 +\widehat{\Psi}(t,x )$, $\widehat{\Psi}(t,x ):=\chi( x_1)\hat{\varphi}(t,x')$.
We also assume that the basic state (\ref{17}) about which we shall linearize problem \eqref{newL}, \eqref{14} satisfies the hyperbolicity condition (\ref{6}) in $\overline{\Omega_T}$,
\begin{equation}
\rho (\hat{p},\widehat{S})  >0,\quad \rho_p(\hat{p},\widehat{S})  >0 ,
\label{19}
\end{equation}
the first boundary condition in (\ref{14}),
\begin{equation}
\partial_t\hat{\varphi}-\hat{v}_{N}|_{x_1=0}=0,
\label{20}
\end{equation}
and the assumption \eqref{16'},
\begin{equation}
\partial_1\hat{p}|_{x_1=0}> - \epsilon ,
\label{21}
\end{equation}
where $\hat{v}_{N}=\hat{v}_1-\hat{v}_2\partial_2\hat{\varphi}-\hat{v}_3\partial_3\hat{\varphi}$.

The linearized equations for \eqref{newL} and (\ref{14}) for determining small perturbations 
$(\delta U,\delta \varphi )$ read (below we drop $\delta$):
\[
\mathbb{L}'(\widehat{U},\widehat{\Psi})(U,\Psi)
:=
L(\widehat{U},\widehat{\Psi})U +{\cal C}(\widehat{U},\widehat{\Psi})
U -   \bigl\{L(\widehat{U},\widehat{\Psi})\Psi\bigr\}\frac{\partial_1(\widehat{U}+\check{U})}{\partial_1\widehat{\Phi}}
=f,
\]
\[
\mathbb{B}'(\widehat{U},\hat{\varphi})(U,\varphi ):=
\left(
\begin{array}{c}
\partial_t\varphi +\hat{v}_2\partial_2\varphi+\hat{v}_3\partial_3\varphi -v_{N}\\[6pt]
p
\end{array}
\right)=g,
\]
where $v_{N}=v_1-v_2\partial_2\hat{\varphi}-v_3\partial_3\hat{\varphi}$, and the matrix
${\cal C}(\widehat{U},\widehat{\Psi})$ is determined as follows:
\begin{multline*}
{\cal C}(\widehat{U},\widehat{\Psi})U
= (U ,\nabla_uA_0(\widehat{U} ))\partial_t\widehat{U}
 +(U ,\nabla_u{A}_{\nu} (\widehat{U},\widehat{\Psi}))\partial_1\widehat{U} \\
+ \sum_{k=2}^3(U ,\nabla_uA_k(\widehat{U} ))\partial_k\widehat{U}
+\left( \begin{array}{c} 0 \\ -g\rho_p(\hat{p},\widehat{S})p -g
\rho_S(\hat{p},\widehat{S})S \\ 0 \\ 0 \\ 0 \end{array}\right)
.\qquad
\end{multline*}
\[
(Y ,\nabla_y A(\widehat{U})):=\sum_{i=1}^5y_i\left.\left(\frac{\partial A (Y )}{
\partial y_i}\right|_{Y =\widehat{U}}\right),\quad Y =(y_1,\ldots ,y_5).
\]
Here, as usual, we introduce the source terms $f=(f_1,\ldots ,f_5)$ and $g=(g_1,g_2)$ to make the interior equations and the boundary conditions inhomogeneous. 

The differential operator $\mathbb{L}'(\widehat{U},\widehat{\Psi})$ is a first order operator in
$\Psi=\chi (x_1)\varphi (t,x')$. Following Alinhac \cite{Al} and introducing the ``good unknown''
\begin{equation}
\dot{U}:=U -\frac{\Psi}{\partial_1\widehat{\Phi}}\,\partial_1(\widehat{U}+\check{U}),
\label{22}
\end{equation}
we simplify the linearized interior equations:
\begin{equation}
L(\widehat{U},\widehat{\Psi})\dot{U} +{\cal C}(\widehat{U},\widehat{\Psi})
\dot{U} - \frac{\Psi}{\partial_1\widehat{\Phi}}\,\partial_1\bigl\{\mathbb{L}
(\widehat{U},\widehat{\Psi})\bigr\}=f.
\label{23}
\end{equation}
As in \cite{Al,CS,Tr1,Tr}, we drop the zero-order term in $\Psi$ in (\ref{23}) and consider the effective linear operators
\begin{equation}
\begin{array}{l}
\mathbb{L}'_e(\widehat{U},\widehat{\Psi})\dot{U} :=L(\widehat{U},\widehat{\Psi}) \dot{U}+{\cal C}(\widehat{U},\widehat{\Psi})\dot{U}
\\[6pt]
\qquad\qquad =A_0(\widehat{U})\partial_t\dot{U}
 +\widetilde{A}_1(\widehat{U},\widehat{\Psi})
\partial_1\dot{U}+A_2(\widehat{U})\partial_2\dot{U}+A_3(\widehat{U})\partial_3\dot{U}
+{\cal C}(\widehat{U},\widehat{\Psi})\dot{U}
\end{array}
\label{24}
\end{equation}
In the subsequent nonlinear analysis the dropped term in (\ref{23}) will be considered as an error term at each Nash-Moser iteration step.

Regarding the boundary differential operator
$\mathbb{B}'$, in terms of unknown (\ref{22}) it reads:
\begin{equation}
\mathbb{B}'_e(\widehat{U},\hat{\varphi})(\dot{U},\varphi ):=
\mathbb{B}'(\widehat{U},\hat{\varphi})(U,\varphi ) =\left(
\begin{array}{c}
\partial_t\varphi+\hat{v}_2\partial_2\varphi+\hat{v}_3\partial_3\varphi-\dot{v}_{N}-
\varphi\,\partial_1\hat{v}_{N}\\[6pt]
\dot{p} + \varphi (2\epsilon +\partial_1\hat{p})
\end{array}\right),
\label{25}
\end{equation}
where $\dot{v}_{\rm N}=\dot{v}_1-\dot{v}_2\partial_2\hat{\varphi}-\dot{v}_3\partial_3\hat{\varphi}$.
Thus, the linear problem for $(\dot{U},\varphi )$ has the form
\begin{align}
\mathbb{L}'_e(\widehat{U},\widehat{\Psi})\dot{U}={f}\quad\mbox{in}\ \Omega_T, \label{26}\\[3pt]
\mathbb{B}'_e(\widehat{U},\hat{\varphi})(\dot{U},\varphi )={g}\quad\mbox{on}\ \partial\Omega_T,\label{27}\\[3pt]
(\dot{U},\varphi )=0\quad \mbox{for}\ t<0,\label{28}
\end{align}
where $f$ and $g$ vanish in the past. We consider the case of zero initial data, that is usual assumption, and postpone the case of nonzero initial data to the nonlinear analysis (construction of a so-called approximate solution).

On the basic state the boundary matrix $\widetilde{A}_1$ has the form
\[
\widetilde{A}_1(\widehat{U},\widehat{\Psi})=\frac{1}{\partial_1\widehat{\Phi}}
\left(
\begin{array}{ccccc} 
{\displaystyle \frac{\hat{\mathfrak{f}}}{\hat{\rho} \hat{c}^2}}&1&-\partial_2\widehat{\Psi}&-\partial_3\widehat{\Psi}&0\\[6pt]
 1&\hat{\rho} \hat{\mathfrak{f}}&0&0&0\\
-\partial_2\widehat{\Psi}&0&\hat{\rho} \hat{\mathfrak{f}}&0&0\\ -\partial_3\widehat{\Psi}&0&0&\hat{\rho} \hat{\mathfrak{f}}&0\\0&0&0&0&\hat{\mathfrak{f}}
\end{array}\right),
\]
where
\[
\hat{\rho}=\rho (\hat{p},\widehat{S}), \quad \hat{c}^2=\rho_p(\hat{p},\widehat{S}),\quad \hat{\mathfrak{f}}=\hat{v}_1-\hat{v}_2\partial_2\widehat{\Psi}-\hat{v}_3\partial_3\widehat{\Psi}-\partial_t\widehat{\Psi}.
\]
In view of (\ref{20}),
\[
\hat{\mathfrak{f}}|_{x_1=0}= \hat{v}_{N}|_{x_1=0}-\partial_t\hat{\varphi}=0.
\]
We see that the boundary matrix $\widetilde{A}_1(\widehat{U},\widehat{\Psi})$
is singular on the boundary $x_1=0$ (it is of constant rank 2 at $x_1=0$). That is, \eqref{26}--\eqref{28} is a hyperbolic problem with {\it characteristic} boundary of constant multiplicity. 

It is convenient to separate ``characteristic'' and ``noncharacteristic'' unknowns. For this purpose we introduce the new unknown
\[
V=(\dot{p},\dot{v}_n,\dot{v}_2,\dot{v}_3,\dot{S}),
\]
where $\dot{v}_n=\dot{v}_1-\dot{v}_2\partial_2\widehat{\Psi}-\dot{v}_3\partial_3\widehat{\Psi}$ ($\dot{v}_n|_{x_1=0}=
\dot{v}_N|_{x_1=0}$). We have $\dot{U}=JV$, with
\[
J=\left(
\begin{array}{ccccc}
1 & \; 0 & 0 & 0 &  0 \\
0 & \; 1 & \;\partial_2\widehat{\Psi} & \partial_3\widehat{\Psi} &  0 \\
0 & \; 0 & 1 & 0 &  0 \\
0 &\; 0 & 0 & 1 &  0 \\
0 &\; 0 & 0 & 0 &  1 \end{array}
\right),
\]
Then, system \eqref{26} is equivalently rewritten as
\begin{equation}
{\cal A}_0(\widehat{U},\widehat{\Psi})\partial_t{V}+ \sum_{k=1}^{3}{\cal A}_k(\widehat{U},\widehat{\Psi})\partial_k{V} +{\cal
A}_4(\widehat{U},\widehat{\Psi}){V} ={\cal F}(\widehat{U},\widehat{\Psi}), \label{29}
\end{equation}
where ${\cal A}_{\alpha}=J^{\textsf{T}}A_{\alpha}J\quad (\alpha =0,2,3),\quad {\cal A}_1=J^{\textsf{T}}\widetilde{A}_1J,\quad {\cal F}=J^{\textsf{T}}f$. The boundary matrix ${\cal A}_1$ in system (\ref{29}) has the form
\begin{equation}
{\cal A}_1=\frac{1}{\partial_1\widehat{\Phi}}{\cal A}_{(1)}+{\cal A}_{(0)},\quad
{\cal A}_{(1)}=\left(\begin{array}{ccccc}
0 & 1 & 0 & 0 & 0 \\
1 & 0 & 0 & 0 & 0 \\
0 & 0 & 0 & 0 & 0 \\
0 & 0 & 0 & 0 & 0 \\
0 & 0 & 0 & 0 & 0
\end{array}
\right),
\quad {\cal A}_{(0)}|_{x_1=0}=0,
\label{30}
\end{equation}
i.e., $V_n=(\dot{p},\dot{v}_n)$ is the ``noncharacteristic'' part of the vector $V$.
The explicit form of ${\cal A}_{(0)}$ is of no interest, and it is only important that, in view (\ref{20}), ${\cal A}_{(0)}|_{x_1=0}=0$. The boundary matrix ${\cal A}_1$ on the boundary $x_1=0$ has one positive (``outgoing'') eigenvalue. Since one of the boundary conditions is needed for determining the function $\varphi$, the correct number of boundary conditions is two (that is the case in \eqref{27}). Hence, the hyperbolic problem  \eqref{26}--\eqref{28} has the property of {\it maximality} \cite{Rauch}.

By standard argument we get for system \eqref{26} the energy inequality
\begin{equation}
I(t)- 2\int_{\partial\Omega_t}\dot{p}\,\dot{v}_N|_{x_1=0}\,{\rm d}x'\,{\rm d}s\leq C(K) 
\left( \| f\|^2_{L_2(\Omega_T)} +\int_0^tI(s)\,{\rm d}s\right),
\label{31}
\end{equation}
where $I(t)=\int_{\mathbb{R}^3_+}({\cal A}_0V,V)\,{\rm d}x$ and $C=C(K)>0$ is a constant depending on $K$ (see \eqref{18}). In view of the boundary conditions \eqref{27}, one has
\[
-2\dot{p}\,\dot{v}_N|_{x_1=0}=2(\varphi\hat{a}- g_2)(\partial_t\varphi+\hat{v}_2\partial_2\varphi+\hat{v}_3\partial_3\varphi-
\varphi\,\partial_1\hat{v}_{N}-g_1)|_{x_1=0}
\]
\[
=\partial_t\left\{ \hat{a}|_{x_1=0}\,\varphi^2
-2g_2\varphi\right\}
-\bigl\{\partial_t\hat{a}+\partial_2(\hat{v}_2\hat{a})+\partial_3(\hat{v}_3\hat{a})-2\hat{a}\partial_1\hat{v}_{N}\bigr\}|_{x_1=0}\,\varphi^2
\]
\[
+2\left\{ \partial_tg_2+\partial_2(\hat{v}_2g_2)+\partial_3(\hat{v}_3g_2)+g_2\partial_1\hat{v}_{N}-g_1\hat{a}\right\}|_{x_1=0}\,\varphi
+2g_1g_2
\]
\[
+\partial_2\left\{ \hat{v}_2\hat{a}\varphi^2 -2\hat{v}_2g_2\varphi \right\}+
\partial_3\left\{ \hat{v}_3\hat{a}\varphi^2 -2\hat{v}_3g_2\varphi  \right\},
\]
where $\hat{a}=2\epsilon +\partial_1\hat{p}$.
Then, using the Young inequality, from \eqref{31} we obtain
\[
\begin{array}{r}
{\displaystyle
I(t)+\int_{\mathbb{R}^2}\left(2\epsilon +\partial_1\hat{p}|_{x_1=0}\right)\varphi^2\,{\rm d}x'
\leq C(K) \Bigl\{ \| f\|^2_{L_2(\Omega_T)}+\| g\|^2_{H^1(\partial\Omega_T)}}\qquad \\[9pt]
{\displaystyle
+\int_0^t\left(I(s)
+\|\varphi (s)\|^2_{L_2(\mathbb{R}^2)}\right)\,{\rm d}s\Bigr\}.}
\end{array}
\]
Taking into account assumption \eqref{21} and applying Gronwall's lemma, we finally deduce the basic a priori $L_2$--estimate
\begin{equation}
\|\dot{U}\|_{L_2(\Omega_T)}+\|\varphi\|_{L_2(\partial\Omega_T)}\leq C(K)\left\{ 
\| f\|_{L_2(\Omega_T)}+\| g\|_{H^1(\partial\Omega_T)}\right\}.
\label{32}
\end{equation}

\begin{remark}{\rm
In the a priori estimate \eqref{32} we have a loss of one derivative from the source term $g$ to the solution
(more precisely, we loose one derivative only from $g_2$ but not from $g_1$). This is quite natural because one can check that the constant coefficients linearized problem, i.e., problem \eqref{26}--\eqref{28} with frozen coefficients satisfies the Kreiss--Lopatinski condition but violates the {\it uniform} Kreiss--Lopatinski condition \cite{Kreiss,Met}.
Although the weak Kreiss--Lopatinski condition holds we had to assume the fulfillment of the extra condition \eqref{21}
while deriving the a priori estimate \eqref{32}. This is very unusual for hyperbolic initial boundary value problems because, as a rule (see, e.g., \cite{CS,Tr}), the fulfillment of the Kreiss--Lopatinski condition is enough for 
obtaining a priori estimates. Actually, in our case the appearance of an extra condition on the level of variable coefficients linear analysis is caused by the fact that the symbol associated to the free surface is {\it not elliptic}, i.e., we are not able to resolve our boundary conditions \eqref{27} for the gradient $(\partial_t\varphi ,\partial_2\varphi ,\partial_3\varphi )$. Therefore, it is also natural that in estimate \eqref{32} we ``lose one derivative from the front'', i.e., we do not have the $H^1$--norm of $\varphi$ in the left-hand side of \eqref{32}.
}\label{r2}
\end{remark}

Since in estimate \eqref{32} we do not lose derivatives from the source term $f$ to the solution, the existence of solutions to problem \eqref{26}--\eqref{28} can be proved by the classical argument of Lax and Phillips \cite{LP}. Indeed, we first reduce our problem to one with homogeneous boundary conditions by subtracting from the solution a more regular function (see, e.g., \cite{RM}). Namely, there exists $\widetilde{U}=(\tilde{p},\tilde{v},\widetilde{S})\in H^{s+1}(\Omega_T)$ vanishing in the past such that 
\[
-\tilde{v}_N=g_1,\quad \tilde{p}=g_2\quad\mbox{on}\ \partial\Omega_T,
\]
where $\tilde{v}_{\rm N}=\tilde{v}_1-\tilde{v}_2\partial_2\hat{\varphi}-\tilde{v}_3\partial_3\hat{\varphi}$.
If $\dot{U}=U^{\natural}+\widetilde{U}$, then $U^{\natural}$ satisfies \eqref{26}--\eqref{28} with $g=0$ and
$f=f^{\natural}$, where $f^{\natural}= f-\mathbb{L}'_e(\widehat{U},\widehat{\Psi})\widetilde{U}$.
That is, it is enough to prove the existence of a solution $(\dot{U},\varphi )$ to problem  \eqref{26}--\eqref{28} with $g=0$. For this problem we have the estimate 
\begin{equation}
\|\dot{U}\|_{L_2(\Omega_T)}+\|\varphi\|_{L_2(\partial\Omega_T)}\leq C(K)\, 
\| f\|_{L_2(\Omega_T)}.
\label{33}
\end{equation}

Having in hand estimate \eqref{33} {\it with no loss of derivatives} we may use the classical argument in \cite{LP}.
In particular, we define a dual problem for \eqref{26}--\eqref{28} as follows:
\begin{align}
 &\mathbb{L}'_e(\widehat{U},\widehat{\Psi})^*\bar{U}=\bar{f}\quad\mbox{in}\ \Omega_T, \label{34}\\[3pt]
& \partial_t\bar{p}+\partial_2(\hat{v_2}\bar{p})+\partial_3(\hat{v_3}\bar{p}) +\bar{p}\partial_1\hat{v}_N+
\bar{v}_N\hat{a}=0\quad\mbox{on}\ \partial\Omega_T,\label{35}\\[3pt]
& \bar{U}=0\quad \mbox{for}\ t<0,\label{36}
\end{align}
where $\bar{U}=(\bar{p},\bar{v},\bar{S})$, $\bar{v}_{\rm N}=\bar{v}_1-\bar{v}_2\partial_2\hat{\varphi}-\bar{v}_3\partial_3\hat{\varphi}$, and
\[
{\mathbb{L}'_e}^*=-\mathbb{L}'_e+{\cal C} +{\cal C}^{\textsf{T}}-\partial_tA_0-\partial_1\widetilde{A}_1
-\partial_2A_2-\partial_3A_3.
\]
Problem \eqref{34}--\eqref{36} is indeed a dual problem for \eqref{26}--\eqref{28} because
for all $\dot{U}\in H^1(\Omega_T)$ and $\bar{U}\in H^1(\Omega_T)$, with $\bar{U}|_{t=T}=0$, satisfying the homogeneous boundary conditions \eqref{27} (with $g=0$) and \eqref{35} respectively, one has
\[
(\mathbb{L}'_e\dot{U},\bar{U})_{L_2(\Omega_T)}-(\dot{U},
{\mathbb{L}'_e}^*\bar{U})_{L_2(\Omega_T)}=-(\widetilde{A}_1\dot{U},\bar{U})_{L_2(\partial\Omega_T)}=
-({\cal A}_1V,\bar{V})_{L_2(\partial\Omega_T)}=0,
\]
where $\bar{V}=J^{-1}\bar{U}$. For the dual problem \eqref{34}--\eqref{36} we can easily get the inequality
\[
\bar{I}(t)+\int_{\mathbb{R}^2}\frac{1}{2\epsilon+\partial_1\hat{p}_{|x_1=0}}\,\bar{p}^2_{|x_1=0}\,{\rm d}x'
\leq C(K) \left\{ \| \bar{f}\|^2_{L_2(\Omega_T)}
+\int_0^t\left(\bar{I}(s)
+\|\bar{p}_{|x_1=0} (s)\|^2_{L_2(\mathbb{R}^2)}\right)\,{\rm d}s\right\}
\]
($\bar{I}(t)=\int_{\mathbb{R}^3_+}({\cal A}_0\bar{V},\bar{V})\,{\rm d}x$) which, in view of condition \eqref{21},
implies the $L_2$--estimate
\[
\|\bar{U}\|_{L_2(\Omega_T)}\leq C(K)\,\| \bar{f}\|_{L_2(\Omega_T)}.
\]
We omit further arguments which are really classical and refer to \cite{LP} (see also, e.g., \cite{ChazPir,Met}).
Thus, we have the following well-posedness theorem for the linearized problem  \eqref{26}--\eqref{28}.

\begin{theorem}
Let assumptions \eqref{18}--\eqref{21} are fulfilled for the basic state (\ref{17}). Then for all $(f,g) \in L_2(\Omega_T)\times H^1(\partial\Omega_T)$ that vanish in the past problem (\ref{26})--(\ref{28}) has a unique solution $(\dot{U},\varphi )\in L_2(\Omega_T)\times L_2(\partial\Omega_T)$. This solution obeys the a priori estimate \eqref{32}.
\label{t3}
\end{theorem}

\begin{remark}{\rm
Strictly speaking, the uniqueness of the solution to problem (\ref{26})--(\ref{28}) follows from estimate (\ref{32}),
provided that our solution belongs to $H^1(\Omega_T)\times H^1(\partial\Omega_T)$. We omit here a formal proof of the existence of solutions having an arbitrary degree of smoothness, and we shall suppose that the existence result of Theorem \ref{t3} is also valid for the function spaces $H^s(\Omega_T)\times H^s(\partial\Omega_T)$, with $s\geq 1$. In this case exact assumptions about the regularity of the basic state will be made in Sect. 3, where we prove a tame a priori estimate in $H^s(\Omega_T)\times H^s(\partial\Omega_T)$ with $s$ large enough.}
\label{r3}
\end{remark}

\section{Tame estimate for the linearized problem}
\label{s4}

We are going to derive a tame a priori estimate in $H^s$ for problem (\ref{26})--(\ref{28}), with $s$ large enough. 
This tame estimate (see Theorem \ref{t4} below) being, roughly speaking, linear in high norms (that are
multiplied by low norms) is with no loss of derivatives from
$f$, with the loss of one derivative from $g$, and with a fixed loss of derivatives with
respect to the coefficients, i.e., with respect to the basic state
(\ref{17}). Although problem (\ref{26})--(\ref{28}) is a hyperbolic problem with characteristic boundary that implies a natural loss of control on derivatives in the normal direction we manage to compensate this loss and derive higher order estimates in usual Sobolev spaces. This is achieved by using the same idea as in \cite{Sch,CS} and estimating missing normal derivatives through a vorticity-type linearized system.

\begin{theorem}
Let $T>0$ and $s\in \mathbb{N}$, with $s\geq 3$. Assume that the basic state $(\widehat{U} ,\hat{\varphi})\in
H^{s+3}(\Omega_T )\times H^{s+3}(\partial\Omega_T)$ satisfies assumptions (\ref{18})--(\ref{21}) and
\begin{equation}
\|\widehat{U}\|_{H^6(\Omega_T )} +\|\hat{\varphi} \|_{H^{6}(\partial\Omega_T)}\leq \widehat{K},
\label{37}
\end{equation}
where $\widehat{K}>0$ is a constant. Let also the data $(f ,g)\in
H^{s}(\Omega_T )\times H^{s+1}(\partial\Omega_T)$ vanish in the past. Then there exists a positive constant $K_0$ that does not depend on $s$ and $T$ and there exists a constant $C(K_0) >0$ such that, if $\widehat{K}\leq K_0$, then there exists a unique solution $(\dot{U} ,\varphi)\in H^{s}(\Omega_T )\times H^{s}(\partial\Omega_T)$ to problem (\ref{26})--(\ref{28}) that obeys the a priori tame estimate
\begin{equation}
\begin{array}{r}
{\displaystyle
\|\dot{U}\|_{H^s(\Omega_T )}+\|\varphi\|_{H^{s}(\partial\Omega_T)}\leq C(K_0)\Bigl\{
\|f\|_{H^{s}(\Omega_T )}+ \|g \|_{H^{s+1}(\partial\Omega_T)} }\qquad\\[9pt]
+\bigl( \|f\|_{H^{3}(\Omega_T )}+ \| g\|_{H^{4}(\partial\Omega_T)} \bigr)\bigl(
\|\widehat{U}\|_{H^{s+3}(\Omega_T )}+\|\hat{\varphi}\|_{H^{s+3}(\partial\Omega_T)}\bigr)\Bigr\}
\end{array}
\label{38}
\end{equation}
for a sufficiently short time $T$.
\label{t4}
\end{theorem}

\begin{proof}
Since arguments below are quite standard we somewhere will drop detailed calculations. By applying to system \eqref{29} the operator 
$\partial^{\alpha}_{\rm tan}=\partial_t^{\alpha_0}\partial_2^{\alpha_2}\partial_3^{\alpha_3}$, with $|\alpha |=
|(\alpha_0,\alpha_2, \alpha_3)|\leq s$, one gets
\begin{equation}
\int_{\mathbb{R}^3_+}({\cal A}_0\partial^{\alpha}_{\rm tan}V,\partial^{\alpha}_{\rm tan}V ) {\rm d}x - 2\int_{\partial\Omega_t}\partial^{\alpha}_{\rm tan}\dot{p}\,\partial^{\alpha}_{\rm tan}{\dot{v}_N}|_{x_1=0}\,{\rm d}x'\,{\rm d}s={\cal R},\label{39}
\end{equation}
where
\[
{\cal R}=\int_{\Omega_t}\Bigl(
\bigl\{{\rm div}{\mathbb A}\,\partial^{\alpha}_{\rm tan}V - 2\sum_{j=0}^3[\partial^{\alpha}_{\rm tan} ,{\cal A}_j]\partial_jV - 2\partial^{\alpha}_{\rm tan}({\cal A}_4V )+2\partial^{\alpha}_{\rm tan}{\cal F}\bigr\},\partial^{\alpha}_{\rm tan}V\Bigr){\rm d}x {\rm d}s ,
\]
${\rm div}{\mathbb A}=\sum_{j=0}^3\partial_j{\cal A}_j$ ($\partial_0:=\partial_t$),  
and we use the notation of commutator: $[a,b]c:=a(bc)-b(ac)$. Using the Moser-type calculus inequalities 
\begin{equation}
\| uv\|_{H^s(\Omega_T)}\leq C\left( \|u\|_{H^s(\Omega_T)}\|v\|_{L_{\infty}(\Omega_T)}
+\|u\|_{L_{\infty}(\Omega_T)}\|v\|_{H^s(\Omega_T)}\right),\label{40}
\end{equation} 
\begin{equation}
\| F(u)\|_{H^s(\Omega_T)}\leq C(M)\left( 1+\|u\|_{H^s(\Omega_T)}\right),\label{41}
\end{equation}
where the function $F$ is a $C^{\infty}$ function of $u$, and $M$ is such a positive constant that
$\|u\|_{L_{\infty}(\Omega_T)}\leq M$, we estimate the right-hand side in \eqref{39}:
\begin{equation}
{\cal R}\leq C(K)\left\{ \|V\|^2_{H^s(\Omega_t)}+\|f\|^2_{H^s(\Omega_T)}+\left(\|\dot{U}\|^2_{W^1_{\infty}(\Omega_T)}+\|f\|^2_{L_{\infty}(\Omega_T)}\right)\left( 1+\|{\rm coeff}\|^2_{s+1}\right)\right\},
\label{42}
\end{equation}
with $\|{\rm coeff}\|_{m}:=\|\widehat{U}\|_{H^{m}(\Omega_T)}+\|\hat{\varphi}\|_{H^{m}(\partial\Omega_T)}$.

Taking into account the boundary conditions, we have: 
\[
-\partial^{\alpha}_{\rm tan}\dot{p}\,\partial^{\alpha}_{\rm tan}{\dot{v}_N}|_{x_1=0}=
2\partial^{\alpha}_{\rm tan}(\varphi\hat{a}- g_2)\partial^{\alpha}_{\rm tan}(\partial_t\varphi+\hat{v}_2\partial_2\varphi+\hat{v}_3\partial_3\varphi-
\varphi\,\partial_1\hat{v}_{N}-g_1)|_{x_1=0}
\]
\[
=\partial_t\left\{ \hat{a}|_{x_1=0}\,(\partial^{\alpha}_{\rm tan}\varphi)^2
-2\partial^{\alpha}_{\rm tan}g_2\,\partial^{\alpha}_{\rm tan}\varphi\right\}+\ldots + 
\underline{\partial_2\left\{ \hat{v}_2|_{x_1=0}[\partial^{\alpha}_{\rm tan},\partial_1\hat{p}|_{x_1=0}]\varphi\right\}\partial^{\alpha}_{\rm tan}\varphi}
+\ldots ,
\]
where the underlined term is just a typical one that gives a biggest loss of derivatives from the coefficients in the final a priori estimate \eqref{38}. Indeed, using the calculus inequality \eqref{40} and the trace theorem, we get
\[
\begin{array}{r}
\|\partial_2\left\{ \hat{v}_2|_{x_1=0}[\partial^{\alpha}_{\rm tan},\partial_1\hat{p}|_{x_1=0}]\varphi\right\}\|^2_{L_2(\mathbb{R}^2)}\leq C(K)\Bigl\{ \|\varphi\|^2_{H^s(\partial\Omega_t)}+\|\varphi\|^2_{L_{\infty}(\partial\Omega_T)}\Bigl( 1
+\|\widehat{U}|_{x_1=0}\|^2_{H^{s+2}(\partial\Omega_T)}\Bigr)\Bigr\}\\[9pt]
\leq
C(K)\Bigl\{ \|\varphi\|^2_{H^s(\partial\Omega_t)}+\|\varphi\|^2_{L_{\infty}(\partial\Omega_T)}
\Bigl( 1+\|\widehat{U}\|^2_{H^{s+3}(\Omega_T)}\Bigr)\Bigr\}.
\end{array}
\]
Omitting detailed calculations, from \eqref{39} and \eqref{42} we obtain 
\begin{equation}
\nt V(t)\nt^2_{{\rm tan},s}+\nt \varphi (t)\|^2_{H^s(\mathbb{R}^2)}\leq C(K){\cal M}(t),
\label{43}
\end{equation}
where
\[
{\cal M}(t)={\cal N}(T)+\int_0^t{\cal I}(s)\,{\rm d}s,\quad
{\cal I}(t)=\nt V(t)\nt^2_{H^s(\mathbb{R}^3_+)}+
\nt \varphi (t)\nt^2_{H^s(\mathbb{R}^2)},
\]
\[
{\cal N}(T)=\|f\|^2_{H^s(\Omega_T)}
+\|g\|^2_{H^{s+1}(\partial\Omega_T)} 
+\left(\|\dot{U}\|^2_{W^1_{\infty}(\Omega_T)}+\|\varphi\|^2_{W^1_{\infty}(\partial\Omega_T)}+\|f\|^2_{L_{\infty}(\Omega_T)}\right)\left( 1+\|{\rm coeff}\|^2_{s+3}\right),
\]
\[
\nt u(t)\nt^2_{{\rm tan},m}:=\sum_{|\alpha|\leq m} \| \partial^{\alpha}_{\rm tan} u (t)\|^2_{L_2(\mathbb{R}^3_+)},\quad
\nt u(t)\nt^2_{H^m(D)}:=
\sum\limits_{j=0}^m\|\partial_t^ju(t)\|^2_{H^m(D)}
\]
($D=\mathbb{R}^2$ or $D=\mathbb{R}^3_+$).
Since only the biggest loss of derivatives from the coefficients will play the role for obtaining the final tame estimate, we have roughened inequality (\ref{43}) by choosing the biggest loss. 

It follows from \eqref{29} and \eqref{30} that
\begin{equation}
(\partial_1V_n,0,0,0)=\bigl(\partial_1\widehat{\Phi}\bigr){\cal A}_{(1)}\Bigl( {\cal F}-{\cal A}_0\partial_tV-\sum_{k=2}^3{\cal A}_k
\partial_kV- {\cal A}_4V-{\cal A}_{(0)}\partial_1V\Bigr). 
\label{44}
\end{equation}
Applying to \eqref{44} the operator $\partial_{\rm tan}^{\beta}$, with $|\beta |\leq s-1$, 
using decompositions like 
\[
\partial_{\rm tan}^{\beta}(B\partial_i V)=B\partial_{\rm tan}^{\beta}\partial_i V + 
[\partial_{\rm tan}^{\beta},B]\partial_i V,
\]
taking into account
the fact that ${\cal A}_{(0)}|_{x_1=0}=0$,  and employing counterparts of the calculus inequalities \eqref{40} and \eqref{41} for the ``layerwise'' norms $\nt (\cdot )(t)\nt$ (see \cite{Sch}), one gets
\begin{equation}
\begin{array}{r}
\|\partial_1\partial_{\rm tan}^{\beta}V_n (t)\|^2_{L_2(\mathbb{R}^3_+)}\leq C(K)\Bigl\{\nt V(t)\nt^2_{{\rm tan},s}
+\|\sigma \partial_1\partial_{\rm tan}^{\beta} V(t)\|^2_{L_2(\mathbb{R}^3_+)}+ \nt V(t) \nt^2_{H^{s-1}(\mathbb{R}^3_+)}
\\[6pt]
+ \nt f (t)\nt^2_{H^{s-1}(\mathbb{R}^3_+)}+
\left(\|\dot{U}\|^2_{W^1_{\infty}(\Omega_T)}+\|f\|^2_{L_{\infty}(\Omega_T)}\right)\left( 1+\nt{\rm coeff (t) }\nt^2_{s+1}\right)\Bigr\},
\end{array}
\label{45}
\end{equation}
where $\sigma=\sigma (x_1)\in C^{\infty}(\mathbb{R}_+)$ is
a monotone increasing function such that $\sigma (x_1)=x_1$ in a neighborhood of
the origin and $\sigma (x_1)=1$ for $x_1$ large enough. Since $\sigma |_{x_1=0}=0$ we do not need to use the boundary conditions to estimate $\sigma\partial_1^j\partial_{\rm tan}^{\gamma}V$, with $j+|\gamma |\leq s$,  and we easily get the inequality
\begin{equation}
\begin{array}{r}
\| \sigma\partial_1^j\partial_{\rm tan}^{\gamma}V(t)\|^2_{L_2(\mathbb{R}^3_+)}
\leq C(K)\Bigl\{ \|V\|^2_{H^s(\Omega_t)}+\|f\|^2_{H^s(\Omega_T)}\qquad
\\[9pt]
+\left(\|\dot{U}\|^2_{W^1_{\infty}(\Omega_T)}+\|f\|^2_{L_{\infty}(\Omega_T)}\right)\left( 1+\|{\rm coeff}\|^2_{s+1}\right)\Bigr\}.
\end{array}
\label{46}
\end{equation}
Taking into account Sobolev's embedding in one space dimension,
\[
\nt u(t) \nt^2_{H^{m-1}(D)}\leq \|u\|^2_{L_{\infty}\left([0,t],H^{m-1}(D)\right)}\leq C 
\|u\|^2_{H^{m}([0,t]\times D)},
\]
and combining \eqref{43}, \eqref{45}, and \eqref{46} for $j=1$, we obtain
\begin{equation}
\nt V(t)\nt^2_{{\rm tan},s}+\nt \varphi (t)\|^2_{H^s(\mathbb{R}^2)}
+\sum_{i=1}^{k}\sum_{|\alpha|\leq s-i}\|\partial_1^i\partial_{\rm tan}^{\alpha}V_n (t)\|^2_{L_2(\mathbb{R}^3_+)}
\leq C(K){\cal M}(t),
\label{47}
\end{equation}
with $k=1$.  

Estimate \eqref{47} for $k=s$ is easily proved by finite induction and equivalently rewritten as
\begin{equation}
\nt V(t)\nt^2_{{\rm tan},s}+\|V_n (t)\|^2_{H^s(\mathbb{R}^3_+)}
+\nt \varphi (t)\|^2_{H^s(\mathbb{R}^2)}\leq C(K){\cal M}(t).
\label{48}
\end{equation}
Missing normal derivatives in \eqref{48} for the ``characteristic'' part $(\dot{v}_2,\dot{v}_3,\dot{S})$ of the unknown $V$ can be estimated from the last equation in \eqref{26},
\begin{equation}
\partial_t\dot{S}+\frac{1}{\partial_1\widehat{\Phi}}\left\{ (\hat{w},\nabla )\dot{S}+(\dot{u},\nabla )\widehat{S}\right\}=f_5,
\label{49}
\end{equation}
and a system for the linearized vorticity $\xi = \nabla \times \tilde{v}$, where
\[
\tilde{v}=(\dot{v}_1,\dot{v}_{\tau_2},\dot{v}_{\tau_3}),\quad
\dot{v}_{\tau_k}=(\dot{v},\tau_k),\quad \tau_2= (\partial_2\hat{\varphi},1,0),\quad 
\tau_3= (\partial_3\hat{\varphi},0,1),
\]
\[
\hat{w}= (\hat{v}_n-\partial_t\widehat{\Psi},\hat{v}_2\partial_1\widehat{\Phi},
\hat{v}_3\partial_1\widehat{\Phi}),\quad \dot{u}=(\dot{v}_n,\dot{v}_2\partial_1\widehat{\Phi},
\dot{v}_3\partial_1\widehat{\Phi}).
\]
This system is obtained by applying the curl operator to the equation for $\tilde{v}$ following from \eqref{26},
\[
\partial_t\tilde{v}+\frac{1}{\partial_1\widehat{\Phi}}\left\{ (\hat{w},\nabla )\tilde{v} +\frac{1}{\rho (\hat{p},\widehat{S})}\,\nabla\dot{p} \right\}+\mbox{l.o.t}=\tilde{f}_v
\]
($\tilde{f}_v=(f_2,f_{\tau_2},f_{\tau_3}),\quad f_{\tau_k}= (f_v,\tau_k),\quad f_v=(f_2,f_3,f_4)$),
and has the form
\begin{equation}
\xi_t +	\frac{1}{\partial_1\widehat{\Phi}}(\hat{w},\nabla )\xi +\mbox{l.o.t}=\nabla\times \tilde{f}_v,
\label{50}	
\end{equation}
where l.o.t. are lower-order terms which exact form has no meaning.

Both equations \eqref{49} and \eqref{50} do not need boundary conditions because, in view of \eqref{20}, the first component of the vector $\hat{w}$ is zero on the boundary $x_1=0$. Therefore, omitting detailed calculations and combining corresponding estimates for the normal derivatives of the ``characteristic'' unknown $(\dot{v}_2,\dot{v}_3,\dot{S})$ with \eqref{48}, we deduce the inequality
\[
{\cal I}(t)\leq C(K)\left\{{\cal N}(T)+\int_0^t{\cal I}(s)\,{\rm d}s\right\}.
\]
Applying then Gronwall's lemma, one gets
\[
{\cal I}(t)\leq C(K)\,e^{C(K)T}{\cal N}(T)
\]
(${\cal I}(0)=0$, see (\ref{28})). Integrating the last inequality over the interval $[0,T]$, we come to the estimate
\begin{equation}
\|V\|^2_{H^s(\Omega_T )}+\|\varphi\|^2_{H^{s}(\partial\Omega_T)}\leq C(K)Te^{C(K)T}{\cal N}(T).
\label{51}
\end{equation}
Recall that $\dot{U}=JV$. Taking into account the decomposition $J(\hat{\varphi})= I +J_0(\hat{\varphi})$ and $J_0(0)=0$, 
using \eqref{40} together with the improved calculus inequality \eqref{41} for the case $F(0)=0$,
\[
\| F(u)\|_{H^s(\Omega_T)}\leq C(M)\|u\|_{H^s(\Omega_T)},
\]
and applying Sobolev's embedding in one space dimension, we obtain
\begin{equation}
\begin{array}{r}
\|\dot{U}\|^2_{H^s(\Omega_T )}= \|V + J_0V\|^2_{H^s(\Omega_T )}\leq C(K)\bigl(
\|V\|^2_{H^s(\Omega_T )} +\|\dot{U}\|^2_{L_{\infty}(\Omega_T)}\|{\rm coeff}\|_s^2\bigr)\qquad \\[9pt]
\leq C(K)\|V\|^2_{H^s(\Omega_T )}+
TC(K)\|\dot{U}\|^2_{L_{\infty}(\Omega_T)}\|{\rm coeff}\|_{s+1}^2.
\end{array}
\label{52}
\end{equation}
Inequalities (\ref{51}) and (\ref{52}) imply
\begin{equation}
\|\dot{U}\|^2_{H^s(\Omega_T )}+\|\varphi\|^2_{H^{s}(\partial\Omega_T)}\leq C(K)Te^{C(K)T}{\cal N}(T).
\label{53}
\end{equation}

Taking into account Theorem \ref{t3} and Remark \ref{r3}, we have the well-posedness of problem (\ref{26})--(\ref{28})
in $H^{s}(\Omega_T )\times H^{s}(\partial\Omega_T)$.
Applying Sobolev's embeddings, from (\ref{53}) with $s\geq 3$ we get
\begin{equation}
\begin{array}{r}
\|\dot{U}\|_{H^s(\Omega_T )}+\|\varphi\|_{H^{s}(\partial\Omega_T)}
\leq C(K)T^{1/2}e^{C(K)T}\Bigl\{ \|f\|_{H^s(\Omega_T)}
+\|g\|_{H^{s+1}(\partial\Omega_T)} \qquad \qquad \\[9pt]
+\left(\|\dot{U}\|_{H^3(\Omega_T)}+\|\varphi\|_{H^3(\partial\Omega_T)}+\|f\|_{H^3(\Omega_T)}\right)
\bigl(
\|\widehat{U}\|_{H^{s+3}(\Omega_T )}+\|\hat{\varphi}\|_{H^{s+3}(\partial\Omega_T)}\bigr)\Bigr\},
\end{array}
\label{54}
\end{equation}
where we have absorbed some norms $\|\dot{U}\|_{H^3(\Omega_T)}$ and $\| \varphi\|_{H^3(\partial\Omega_T)}$ in the left-hand side by choosing $T$ small enough. Considering (\ref{54}) for $s=3$ and using (\ref{37}), we obtain for $T$ small enough that
\begin{equation}
\|\dot{U}\|_{H^3(\Omega_T )}+\|\varphi\|_{H^{3}(\partial\Omega_T)}\leq C(K_0)
\left\{\|f\|_{H^3(\Omega_T)}+\|g\|_{H^{4}(\partial\Omega_T)}\right\}.
\label{55}
\end{equation}
It is natural to assume that $T<1$ and, hence, we can suppose that the constant $C(K_0)$ does not depend on $T$.
Inequalities (\ref{54}) and (\ref{55}) imply (\ref{38}).
\end{proof}

\section{Nash-Moser iteration}
\label{s5}

To use the tame estimate (\ref{38}) for the proof of convergence of the Nash-Moser iteration, we should reduce our nonlinear problem \eqref{newL}, \eqref{14}, \eqref{15} on $[0,T]\times\mathbb{R}^3_+$ to that on $\Omega_T$ which solutions vanish in the past. This is achieved by the classical argument suggesting to absorb the initial data into the interior equations by constructing a so-called {\it approximate solution}. Before constructing the approximate solution we have to define {\it compatibility conditions} for the initial data \eqref{15},
\[(U_0,\varphi_0)=(p_0,v_{1,0},v_{2,0},v_{3,0},S_{0},\varphi_0).
\]

Assuming that the hyperbolicity condition \eqref{6} is satisfied, we rewrite system \eqref{newL} in the form
\begin{equation}
\partial_t U = -(A_0(U ))^{-1}\left( \widetilde{A}_1(U , {\Psi})\partial_1U +
A_2(U )\partial_2U+A_3(U )\partial_3U+{A}_{\nu} (U,\Psi)\partial_1\check{U}+Q(U)\right).
\label{56}
\end{equation}
The traces
\[
U_j=(p_j,v_{1,j},v_{2,j},v_{3,j},S_{j})=\partial_t^jU|_{t=0}\quad \mbox{and}\quad \varphi_j=\partial_t^j\varphi|_{t=0},
\]
with $j\geq 1$, are recursively defined by the formal application of the differential operator $\partial_t^{j-1}$ 
to the boundary condition
\begin{equation}
\partial_t\varphi =\left(v_1-v_2\partial_2\varphi -v_3\partial_3\varphi\right)|_{x_1=0}
\label{57}
\end{equation}
and (\ref{56}) and evaluating $\partial_t^j\varphi$ and $\partial_t^jU$ at $t=0$. Moreover,
$\Psi_j=\partial_t^j\Psi|_{t=0}=\chi (x_1)\varphi_j$. 

We naturally define the zero-order compatibility condition
as $p_0|_{x_1=0}=0$. Note that, unlike the case when the symbol associated with the free surface is elliptic \cite{CS,Met,Tr}, this condition does not contain the function $\varphi_0$. Evaluating \eqref{57} at $t=0$, we get
\begin{equation}
\varphi_1 =\left(v_{1,0}-v_{2,0}\partial_2\varphi_0 -v_{3,0}\partial_3\varphi_0\right)|_{x_1=0},
\label{58}
\end{equation}
and then, with $\partial_t\Phi|_{t=0}:=\Phi_1=\chi(x_1)\varphi_1$, from \eqref{56} evaluated at $t=0$ we define $U_1$.
The first-order compatibility condition $p_1|_{x_1=0}=0$ will implicitly depend on $\varphi_0$ and $\varphi_1$.
Knowing $\varphi_1$ and $U_1$ we can then find $\varphi_2$, $U_2$, etc.  
The following lemma is the analogue of Lemma 4.2.1 in \cite{Met}, Lemma 2 in \cite{CS}, and Lemma 5 in \cite{Tr}.

\begin{lemma}
Let $\mu\in\mathbb{N}$, $\mu \geq 3$, $U_0\in H^{\mu}(\mathbb{R}^3_+)$, and $\varphi_0\in H^{\mu }(\mathbb{R}^2)$. Then, the procedure described above determines $U_j\in H^{\mu -j}(\mathbb{R}^3_+)$ and $\varphi_j\in H^{\mu -j}(\mathbb{R}^2)$ for $j= 1,\ldots ,\mu$. Moreover,
\begin{equation}
\sum_{j=1}^{\mu}\left( \|U_j\|_{H^{\mu -j}(\mathbb{R}^3_+)}+\|\varphi_j\|_{H^{\mu -j}(\mathbb{R}^2)} \right)
\leq CM_0,
\label{59}
\end{equation}
where
\begin{equation}
M_0=\|U_0\|_{H^{\mu}(\mathbb{R}^3_+)}+\|\varphi_0\|_{H^{\mu}(\mathbb{R}^2)},
\label{60}
\end{equation}
the constant $C>0$ depends only on $\mu$ and the norms $\|U_0\|_{W^{1}_{\infty}(\mathbb{R}^3_+)}$ and
$\|\varphi_0\|_{W^1_{\infty}(\mathbb{R}^2_+)}$.
\label{l1}
\end{lemma}

The proof is almost evident and based on the multiplicative properties of Sobolev spaces (Remark \ref{rcheck} should be also taken into account).

\begin{definition}
Let $\mu\in\mathbb{N}$, $\mu \geq 3$. The initial data
$
(U_0,\varphi_0) \in H^{\mu }(\mathbb{R}^3_+)\times H^{\mu }(\mathbb{R}^2)
$
are said to be compatible up to order $\mu$ when $(U_j,\varphi_j)$ satisfy 
\begin{equation}
p_j|_{x_1=0}=0
\label{61}
\end{equation}
for $j=0,\ldots ,\mu$.
\label{d1}
\end{definition}

We are now ready to construct the approximate solution.

\begin{lemma}
Suppose the initial data (\ref{15}) are compatible up to order $\mu$ and satisfy the assumptions of Theorem \ref{t2} (i.e., (\ref{6}) for all $x\in\overline{\mathbb{R}^3_+}$ and (\ref{16'})). Then there exists a vector-function
$(U^{a},\varphi^a)\in H^{\mu +1}(\Omega_T)\times H^{\mu +1}(\partial\Omega_T)$,
that is further called the approximate solution to problem \eqref{newL}, \eqref{14}, \eqref{15}, such that
\begin{equation}
\partial_t^j\mathbb{L}(U^a,\Psi^a )|_{t=0}=0 \quad \mbox{for}\ j=0,\ldots , \mu -1,
\label{62}
\end{equation}
and it satisfies the boundary conditions (\ref{14}), where ${\Psi}^a =\chi (x_1)\varphi^a$. Moreover, the approximate solution obeys the estimate
\begin{equation}
\|U^a\|_{H^{\mu +1}(\Omega_T)}+\|\varphi^a\|_{H^{\mu +1}(\partial\Omega_T)}\leq C_1(M_0)
\label{63}
\end{equation}
and satisfies the hyperbolicity condition (\ref{6}) on $\overline{\Omega_T}$  as well as condition (\ref{16'}) on $\partial\Omega_T$, where $C_1=C_1(M_0)>0$ is a constant depending on $M_0$ (see (\ref{60})). Moreover, $\rho^a-\epsilon_1 =\rho (p^a,S^a)-\epsilon_1\in H^{\mu +1}(\Omega_T)$.
\label{l2}
\end{lemma}

\begin{proof} 
Consider functions $U^a \in H^{\mu +1}(\mathbb{R}\times\mathbb{R}^3_+)$ and $\varphi^a\in H^{\mu +1}(\mathbb{R}^3)$ such that
\[
\partial_t^jU^a|_{t=0}=U_j\in H^{\mu -j}(\mathbb{R}^3_+),\quad
\partial_t^j\varphi^a|_{t=0}=\varphi_j\in H^{\mu-j}(\mathbb{R}^2)
\quad \mbox{for}\ j=0,\ldots ,\mu ,
\]
where $U_j$ and $\varphi_j$ are given by Lemma \ref{l1}. Thanks to \eqref{58} and \eqref{61} we can choose $U^a $ and $\varphi^a$ that satisfy the boundary conditions (\ref{14}). By using a cut-off $C^{\infty}_0$ function we can suppose that $(U^a,\varphi^a)$ vanishes outside of the interval $[-T,T]$, i.e., $(U^{a},\varphi^a)\in H^{\mu +1}(\Omega_T)\times H^{\mu +1}(\partial\Omega_T)$. Applying Sobolev's embeddings, we rewrite estimate (\ref{59}) as
\begin{equation}
\sum_{j=1}^{\mu}\left( \|U_j\|_{H^{\mu -j}(\mathbb{R}^3_+)}+\|\varphi_j\|_{H^{\mu -j}(\mathbb{R}^2)} \right)
\leq C(M_0),
\label{64}
\end{equation}
where $C=C(M_0)>0$ is a constant depending on $M_0$. The estimate (\ref{63}) follows from (\ref{64}) and the continuity of the lifting operators from the hyperplane $t=0$ to $\mathbb{R}\times\mathbb{R}^3_+$.
Conditions (\ref{62}) hold thanks to the properties of $(U_j,\varphi_j)$ given by Lemma \ref{l1}.
At last, since $(U^{a},\varphi^a)$ satisfies the hyperbolicity condition (\ref{6}) and  condition (\ref{16'})
at $t=0$, in the above procedure we can choose $(U^{a},\varphi^a)$ that it satisfies (\ref{6}) and (\ref{16'}) for all times $t\in[-T, T]$. The condition $\rho^a-\epsilon_1 \in H^{\mu +1}(\Omega_T)$ is just an assumption on the state equation.
\end{proof}

Without loss of generality we can suppose that
\begin{equation}
\|U_0\|_{H^{\mu}(\mathbb{R}^3_+)}+\|\varphi_0\|_{H^{\mu}(\mathbb{R}^2)}\leq 1,\quad 
\|\varphi_0\|_{H^{\mu}(\mathbb{R}^2)}\leq 1/2.
\label{65}
\end{equation}
Then for a sufficiently short time interval $[0,T]$ the smooth solution which existence we are going to prove satisfies
$\|\varphi\|_{L_{\infty}([0,T]\times\mathbb{R}^2)}\leq 1$ that implies $\partial_1\Phi\geq 1/2$
(recall that $\|\chi'\|_{L_{\infty}(\mathbb{R})}<1/2$, see Section \ref{s3}). Let $\mu$ is an integer number that will appear in the regularity assumption for the initial data in the existence theorem for problem \eqref{newL}, \eqref{14}, \eqref{15}. Running ahead, we take $\mu=m+7$, with $m\geq 6$ (see Theorem \ref{t2}).
In the end of this section  we will see that this choice is suitable. Taking into account (\ref{65}),
we rewrite (\ref{63}) as
\begin{equation}
\|U^a\|_{H^{m +8}(\Omega_T)}+\|\varphi^a\|_{H^{m +8}(\partial\Omega_T)}\leq C_*,
\label{66}
\end{equation}
where $C_*=C_1(1)$.

Let us introduce
\begin{equation}
f^a:=\left\{ \begin{array}{cr}
- \mathbb{L}(U^a,{\Psi}^a ) & \quad \mbox{for}\ t>0,\\
0 & \ \mbox{for}\ t<0.\end{array}\right.
\label{67}
\end{equation}
Since $(U^{a},\varphi^a)\in H^{m +8}(\Omega_T)\times H^{m +8}(\partial\Omega_T)$, using (\ref{62}), we get
$f^a \in H^{m +7}(\Omega_T)$ and
\begin{equation}
\|f^a\|_{H^{m + 7}(\Omega_T)}\leq \delta_0 (T),
\label{68}
\end{equation}
where the constant $\delta_0(T)\rightarrow 0$ as $T\rightarrow 0$. The crucial role in the proof of the fact that $f^a$ belongs to a Sobolev space is played by the {\it presence of gravity} (see Remark \ref{rcheck}). 
To prove estimate (\ref{68}) we use the Moser-type and embedding inequalities and the fact that $f^a$ vanishes in the past.
Then, given the approximate solution defined in Lemma \ref{l2}, $(U ,\varphi)= (U^a,\varphi^a)+ (\widetilde{U} ,\tilde{\varphi})$ is a solution of the original problem \eqref{newL}, \eqref{14}, \eqref{15}
on $[0,T]\times\mathbb{R}^3_+$ if $(\widetilde{U} ,\tilde{\varphi})$ satisfies the following problem on $\Omega_T$ (tildes are dropped):
\begin{align}
{\cal L}(U ,{\Psi})=f^a\quad\mbox{in}\ \Omega_T, \label{69}\\[3pt]
{\cal B}(U ,\varphi )=0\quad\mbox{on}\ \partial\Omega_T,\label{70}
\\[3pt]
(U,\varphi )=0\quad \mbox{for}\ t<0,\label{71}
\end{align}
where ${\cal L}(U ,{\Psi}):=\mathbb{L}(U^a +U ,{\Psi}^a+{\Psi}) -
\mathbb{L}(U^a ,{\Psi}^a),\quad
{\cal B}(U ,\varphi):=\mathbb{B}(U^a +U ,\varphi^a+\varphi )$.
From now on we concentrate on the proof of the existence of solutions to problem (\ref{69})--(\ref{71}).

We solve problem (\ref{69})--(\ref{71}) by a suitable Nash-Moser-type iteration scheme. In short, this scheme is a modified Newton's scheme and at each Nash-Moser iteration step we smooth the coefficient $u_n$ of a corresponding linear
problem for $\delta u_n =u_{n+1}-u_n$. Errors of a classical Nash-Moser iteration are the ``quadratic'' error of Newton's scheme and the ``substitution'' error caused by the application of smoothing operators $S_{\theta}$  (see, e.g., \cite{Herm} and references therein). As in \cite{CS,Tr}, in our case the Nash-Moser procedure is not completely standard and we have the additional error caused by the introduction of an intermediate (or modified) state $u_{n+1/2}$ satisfying some nonlinear constraints. In our case, the main constraint is condition \eqref{20} that was required to be fulfilled for the basic state (\ref{17}). Also the additional error is caused by dropping the zero-order term in $\Psi$ in the  linearized interior equations written in terms of the ``good unknown'' (see (\ref{22})--(\ref{24})). We first list the important properties of smoothing operators \cite{Al,CS,Herm}.

\begin{proposition}
There exists such a family $\{S_{\theta}\}_{\theta\geq 1}$ of smoothing operators in $H^s(\Omega_T)$ acting on the class of functions vanishing in the past that
\begin{align}
\|S_{\theta}u\|_{H^{\beta}(\Omega_T) }\leq C\theta^{(\beta-\alpha )_+}\|u\|_{H^{\alpha}(\Omega_T) }, & \qquad \alpha ,\beta \geq 0,
\label{72}\\
\|S_{\theta}u-u\|_{H^{\beta}(\Omega_T) }\leq C\theta^{\beta-\alpha }\|u\|_{H^{\alpha}(\Omega_T) }, & \qquad 0\leq \beta \leq \alpha ,\label{73}\\
\bigl\|\frac{d}{d\theta}S_{\theta}u\bigr\|_{H^{\beta}(\Omega_T) }\leq C\theta^{\beta-\alpha -1}\|u\|_{H^{\alpha}(\Omega_T) }, & \qquad \alpha ,\beta \geq 0,\label{74}
\end{align}
where $C>0$ is a constant, and $(\beta-\alpha )_+:=\max (0,\beta -\alpha )$.  Moreover, there is another family of smoothing operators (still denoted
$S_{\theta}$) acting on functions defined on the boundary $\partial\Omega_T$ and meeting properties (\ref{72})--(\ref{74}), with the norms $\|\cdot \|_{H^{\alpha}(\partial\Omega_T)}$.
\label{p1}
\end{proposition}

Now, following \cite{CS,Tr}, we describe the iteration scheme for problem (\ref{69})--(\ref{71}). We choose
\[
U_0 =0,\quad \varphi_0=0
\]
and assume that $(U_k, \varphi_k)$ are already given for $k=0,\ldots ,n$.
Moreover, let $(U_k,\varphi_k)$ vanish in the past, i.e., they satisfy (\ref{71}).  We define
\[
U_{n+1}=U_n+\delta U_n,\quad \varphi_{n+1}=\varphi_n+\delta \varphi_n,
\]
where the differences $\delta{U}_n$ and $\delta \varphi_n$ solve the linear problem
\begin{equation}
\left\{\begin{array}{lr}
\mathbb{L}'_e({U}^a +{U}_{n+1/2} ,{\Psi}^a+{\Psi}_{n+1/2})\delta\dot{{U}}_n={f}_n &\quad\mbox{in}\ \Omega_T, \\[6pt]
\mathbb{B}'_{n+1/2}(\delta\dot{{U}}_n,\delta {\varphi}_n)={g}_n  & \quad\mbox{on}\ \partial\Omega_T,\\[6pt]
(\delta\dot{{U}}_n,\delta \varphi_n)=0 &\quad \mbox{for}\ t<0.
\end{array}\right.
\label{75}
\end{equation}
Here
\begin{equation}
\delta\dot{U}_n:= \delta{U}_n-\frac{\delta\Psi_n}{\partial_1(\Phi^a+\Psi_{n+1/2})}\,\partial_1(\check{U}+{U}^a+{U}_{n+1/2})
\label{76}
\end{equation}
is the ``good unknown'' (cf. (\ref{22})),
\[
\mathbb{B}'_{n+1/2}:=\mathbb{B}'_e(({U}^a +{U}_{n+1/2})|_{x_1=0} ,\varphi^a+\varphi_{n+1/2}),
\]
the operators $\mathbb{L}'_e$ and $\mathbb{B}'_e$ are defined in (\ref{24}), (\ref{25}), and $({U}_{n+1/2},\varphi_{n+1/2})$ is a smooth modified state such that $({U}^a +{U}_{n+1/2},\varphi^a+\varphi_{n+1/2})$ satisfies constraints (\ref{19})--(\ref{21}) ($\Psi_n$, ${\Psi}_{n+1/2}$, and $\delta\Psi_n$ are associated to $\varphi_n$, $\varphi_{n+1/2}$, and $\delta\varphi_n$ like ${\Psi}$ is associated to $\varphi$). The right-hand sides ${f}_n$ and ${g}_n$  are defined through the accumulated errors at the step $n$.

The errors of the iteration scheme are defined from the following chains of decompositions:
\[
\begin{array}{l}
{\cal L}({U}_{n+1} ,{\Psi}_{n+1})-{\cal L}({U}_{n} ,{\Psi}_{n})\\
\quad =\mathbb{L}'({U}^a +{U}_{n} ,{\Psi}^a+{\Psi}_{n})(\delta{{U}}_n,\delta{\Psi}_{n})+{e}'_n\\
\quad =\mathbb{L}'({U}^a +S_{\theta_n}{U}_{n} ,{\Psi}^a+S_{\theta_n}{\Psi}_{n})(\delta{{U}}_n,\delta{\Psi}_{n})+{e}'_n+{e}''_n
\\
\quad= \mathbb{L}'({U}^a +{U}_{n+1/2} ,{\Psi}^a+{\Psi}_{n+1/2})(\delta{{U}}_n,\delta{\Psi}_{n})+{e}'_n+{e}''_n+{e}'''_n\\
\quad =\mathbb{L}'_e({U}^a +{U}_{n+1/2},{\Psi}^a+{\Psi}_{n+1/2})\delta\dot{{U}}_n+{e}'_n+{e}''_n+{e}'''_n+{D}_{n+1/2}\delta{\Psi}_{n}
\end{array}
\]
and
\[
\begin{array}{l}
{\cal B}({U}_{n+1}|_{x_1=0},\varphi_{n+1})-{\cal B}({U}_{n}|_{x_1=0},\varphi_{n})
\\ \quad =\mathbb{B}'(({U}^a +{U}_{n})|_{x_1=0},\varphi^a+\varphi_{n})(\delta{{U}}_n|_{x_1=0},\delta \varphi_{n})+\tilde{e}'_n
\\ \quad =\mathbb{B}'(({U}^a +S_{\theta_n}{U}_{n})|_{x_1=0},\varphi^a+S_{\theta_n}\varphi_{n})(\delta{{U}}_n|_{x_1=0},\delta \varphi_{n})+\tilde{e}'_n+\tilde{e}''_n
\\ \quad =\mathbb{B}'_{n+1/2}(\delta\dot{{U}}_n,\delta\varphi_n)+\tilde{e}'_n+\tilde{e}''_n+\tilde{e}'''_n,
\end{array}
\]
where $S_{\theta_n}$ are smoothing operators enjoying the properties of Proposition \ref{p1}, with the sequence $(\theta_n)$ defined by
\[
\theta_0\geq 1,\quad \theta_n=\sqrt{\theta_0+n} ,
\]
and we use the notation
\[
{D}_{n+1/2}:= \frac{1}{\partial_1(\Phi^a+\Psi_{n+1/2})}\,\partial_1\left\{ \mathbb{L}(U^a +U_{n+1/2} ,{\Psi}^a+{\Psi}_{n+1/2})\right\}.
\]
The errors ${e}'_n$ and  $\tilde{e}'_n$ are the usual quadratic errors of Newton's method, and ${e}''_n$, 
$\tilde{ e}''_n$ and ${e}'''_n$, $\tilde{e}'''_n$ are the first and the second substitution errors
respectively.

Let
\begin{equation}
{e}_n:={e}'_n+{e}''_n+{e}'''_n+{D}_{n+1/2}\delta{\Psi}_{n}, \quad
\tilde{e}_n:= \tilde{e}'_n+\tilde{e}''_n+\tilde{e}'''_n,
\label{77}
\end{equation}
then the accumulated errors at the step $n\geq 1$ are
\begin{equation}
{E}_n=\sum_{k=0}^{n-1}{e}_k,\quad \widetilde{E}_n=\sum_{k=0}^{n-1}\tilde{e}_k,
\label{78}
\end{equation}
with ${E}_0:=0$ and $\widetilde{E}_0:=0$. The right-hand sides ${f}_n$ and ${g}_n$ are recursively computed from the equations
\begin{equation}
\sum_{k=0}^{n}{f}_k+S_{\theta_n}{E}_n=S_{\theta_n}{f}^a,\quad
\sum_{k=0}^{n}{g}_k+S_{\theta_n}\widetilde{E}_n=0,
\label{79}
\end{equation}
where ${f}_0:=S_{\theta_0}{f}^a$ and ${g}_0:=0$. Since $S_{\theta_N}\rightarrow I$ as $N\rightarrow \infty$, one can show that we formally obtain
the solution to problem (\ref{69})--(\ref{71}) from ${\cal L}(U_{N} ,{\Psi}_{N})\rightarrow {f}^a$ and
${\cal B}(U_{N}|_{x_1=0},\varphi_{N})\rightarrow 0$, provided that $({e}_N,\tilde{e}_N)\rightarrow 0$.

\begin{remark}{\rm 
In general, the realization of the Nash-Moser procedure for problem (\ref{69})--(\ref{71}) below is much simpler as in \cite{Tr} for current-vortex sheets.  As in \cite{CS} and unlike \cite{Tr}, we work in usual Sobolev spaces $H^s$
(in \cite{Tr} one works in the anisotropic weighted Sobolev spaces $H^s_*$). More precisely, in \cite{CS} the exponentially weighted Sobolev spaces $H^s_{\gamma}:=e^{\gamma t}H^s$ were used, but for $H^s_{\gamma}$,
Sobolev's embeddings, Moser-type inequalities, etc. are internally the same as for the usual Sobolev spaces $H^s$. Therefore, in some places below our calculations are almost the same as in \cite{CS}. However, for convenience of the reader we prefer to present all the calculations (at least, in brief). Moreover, since, unlike \cite{CS}, we do not assume that our initial data are close to a constant solution  and in our tame estimate \eqref{38} we lose, as \cite{Tr}, ``one derivative from the front'', somewhere we have to modify arguments of \cite{CS}.
}
\end{remark}

Below we closely follow the plan of \cite{CS} and \cite{Tr}. Let us first formulate an inductive hypothesis.
As in \cite{Tr} and unlike \cite{CS}, we do not require more regularity for $\delta \varphi_k$ in our inductive hypothesis.

\paragraph{Inductive hypothesis.} Given a small number $\delta >0$, the integer $\alpha :=m+1$, and an integer $\tilde{\alpha}$, our inductive hypothesis reads:
\[
(H_{n-1})\quad \left\{
\begin{array}{ll}
a)\;& \forall\, k=0,\ldots , n-1,\quad \forall s\in [3,\tilde{\alpha}]\cap\mathbb{N},\\[3pt]
 & \|\delta U_k\|_{H^s(\Omega_T)} +\|\delta \varphi_k\|_{H^s(\partial\Omega_T)}\leq \delta\theta_k^{s-\alpha -1}\Delta_k,\\[6pt]
b) & \forall\, k=0,\ldots , n-1,\quad \forall s\in [3,\tilde{\alpha}-2]\cap\mathbb{N},\\[3pt]
 & \|{\cal L}(U_k,{\Psi}_k)-f^a\|_{H^s(\Omega_T)}\leq 2\delta\theta_k^{s-\alpha -1},\\[6pt]
c) & \forall\, k=0,\ldots , n-1,\quad \forall s\in [4,\alpha ]\cap\mathbb{N},\\[3pt]
 & \|{\cal B}(U_k|_{x_1=0},\varphi_k)\|_{H^s(\partial\Omega_T)}\leq \delta\theta_k^{s-\alpha -1},
\end{array}\right.
\]
where $\Delta_k= \theta_{k+1}-\theta_k$. Note that the sequence $(\Delta_n)$ is decreasing and tends to zero, and
\[
\forall\, n\in\mathbb{N},\quad \frac{1}{3\theta_n}\leq\Delta_n=\sqrt{\theta_n^2+1} -\theta_n\leq \frac{1}{2\theta_n}.
\]
Recall that $(U_k,\varphi_k)$ for $k=0,\ldots ,n$ are also assumed to satisfy (\ref{71}). Running a few steps forward, we observe that we will need to use inequalities (\ref{66}) and (\ref{68}) with $m=\tilde{\alpha}-4$. That is, we now choose $\tilde{\alpha}=m+4$. 
Our goal is to prove that ($H_{n-1}$) implies ($H_n$) for a suitable choice of parameters $\theta_0\geq 1$ and $\delta >0$, and for a sufficiently short time $T>0$. After that we shall prove ($H_0$). From now on we assume that ($H_{n-1}$) holds.  As in \cite{CS}, we have the following consequences of ($H_{n-1}$).

\begin{lemma}
If $\theta_0$ is big enough, then for every $k=0,\ldots ,n$ and for every integer $s\in [3,\tilde{\alpha}]$ we have
\begin{align}
\|{U}_k \|_{H^s(\Omega_T)}+\| \varphi_k\|_{H^{s}(\partial\Omega_T)}\leq\delta\theta_k^{(s-\alpha )_+}, &\quad \alpha\neq s,\label{80}\\[3pt]
\|{U}_k \|_{H^{\alpha}(\Omega_T)}+\| \varphi_k\|_{H^{\alpha}(\partial\Omega_T)}\leq\delta\log \theta_k, & \label{81}
\end{align}
\begin{equation}
\|(I-S_{\theta_k}){U}_k \|_{H^s(\Omega_T)}+\|(1-S_{\theta_k}) \varphi_k\|_{H^{s}(\partial\Omega_T)}\leq C\delta\theta_k^{s-\alpha }.\label{82}
\end{equation}
For every $k=0,\ldots ,n$ and for every integer $s\in [3,\tilde{\alpha}+4]$ we have
\begin{align}
\|S_{\theta_k}{U}_k \|_{H^s(\Omega_T)}+\| S_{\theta_k} \varphi_k\|_{H^{s}(\partial\Omega_T)}\leq C\delta\theta_k^{(s-\alpha )_+}, &\quad \alpha\neq s,\label{83}\\[3pt]
\|S_{\theta_k}{U}_k \|_{H^{\alpha}(\Omega_T)}+\| S_{\theta_k}\varphi_k\|_{H^{\alpha}(\partial\Omega_T)}\leq C\delta\log \theta_k. &  \label{84}
\end{align}
\label{l3}
\end{lemma}

Estimates (\ref{82})--(\ref{84}) follow from (\ref{80}), (\ref{81}), and Proposition \ref{p1}. Moreover, (\ref{82}) and (\ref{83}) hold actually for every integer $s\geq 3$ but below we will need them only for $s\in [3,\tilde{\alpha}]$ and $s\in [3,\tilde{\alpha}+4]$ respectively.

\paragraph{Estimate of the quadratic errors.} The quadratic errors 
\[
{e}'_k= {\cal L}(U_{k+1} ,{\Psi}_{k+1})-{\cal L}(U_{k} ,{\Psi}_{k})
-{\cal L}'(U_{k} ,{\Psi}_{k})(\delta{U}_k,\delta{\Psi}_{k}),
\]
\[
\tilde{e}'_k= \bigl({\cal B}(U_{k+1},\varphi_{k+1})-{\cal B}(U_{k},\varphi_{k})
-{\cal B}'(U_{k},\varphi_{k})(\delta{U}_k,\delta \varphi_{k})\bigr)|_{x_1=0}
\]
can be rewritten as
\begin{equation}
{e}'_k=\int_0^1(1-\tau )\mathbb{L}''({U}^a+{U}_k +\tau\delta{U}_k, {\Psi}^a+{\Psi}_k
+\tau\delta{\Psi}_k)\bigl((\delta{U}_k,\delta{\Psi}_k),
(\delta{U}_k,\delta{\Psi}_k)\bigr) d\tau ,
\label{85}
\end{equation}
\begin{equation}
\tilde{e}'_k=\frac{1}{2}\,\mathbb{B}''\bigl(
(\delta{U}_k|_{x_1=0},\delta \varphi_k),(\delta{U}_k|_{x_1=0},\delta \varphi_k)\bigr) 
\label{86}
\end{equation}
by using the second derivatives of the operators $\mathbb{L}$ and $\mathbb{B}$:
\[
\mathbb{L}''(\widehat{U},\widehat{{\Psi}})((U',{\Psi}'),(U'',{\Psi}'')):=
\frac{\rm d}{{\rm d}\varepsilon}\mathbb{L}'(U_{\varepsilon},{\Psi}_{\varepsilon})(U',{\Psi}')|_{\varepsilon =0}
\qquad (\mathbb{L}'(\widehat{U},\widehat{{\Psi}})(U'',{\Psi}''):=\frac{\rm d}{{\rm d}\varepsilon}\mathbb{L}
(U_{\varepsilon},{\Psi}_{\varepsilon})),
\]
\[
 \mathbb{B}''((W',\varphi'),(W'',\varphi'')):=
\frac{\rm d}{{\rm d}\varepsilon}\mathbb{B}'(W_{\varepsilon},\varphi_{\varepsilon})(W',\varphi')|_{\varepsilon =0}
\qquad (\mathbb{B}'(\widehat{U}|_{x_1=0},\hat{\varphi})(W'',{\varphi}'')=\frac{\rm d}{{\rm d}\varepsilon}\mathbb{B}(W_{\varepsilon},{\varphi}_{\varepsilon}) ),
\]
where $U_{\varepsilon}=\widehat{U}+\varepsilon U''$, $W_{\varepsilon}=\widehat{U}|_{x_1=0}+\varepsilon W''$,
$\varphi_{\varepsilon}=\hat{\varphi}+{\varepsilon}\varphi''$, and ${\Psi}'$ and ${\Psi}''$ are associated to
$\varphi'$ and $\varphi''$ respectively like ${\Psi}$ is associated to $\varphi$. We easily compute the explicit form of $\mathbb{B}''$, that do not depend on the state $(\widehat{U},\hat{\varphi})$:
\[
\mathbb{B}''((W',\varphi'),(W'',\varphi''))=\left(
\begin{array}{c}
v'_2\partial_2\varphi''+v'_3\partial_3\varphi''+v''_2\partial_2\varphi'+v''_3\partial_3\varphi'\\
0
\end{array}
\right).
\]
To estimate the quadratic errors by utilizing representations (\ref{85}) and (\ref{86}) we need estimates for
$\mathbb{L}''$ and $\mathbb{B}''$. They can easily be obtained from the explicit forms of $\mathbb{L}''$ and $\mathbb{B}''$ by applying the Moser-type and embedding inequalities. Omitting detailed calculations, we get the following result.

\begin{proposition}
Let $T>0$ and $s\in\mathbb{N}$, with $s\geq 3$. Assume that $(\widehat{U} ,\hat{\varphi})\in
H^{s+1}(\Omega_T )\times H^{s+1}(\partial\Omega_T)$ and
\[
\|\widehat{U}\|_{H^3((\Omega_T)} +\|\hat{f} \|_{H^{3}(\partial\Omega_T)}\leq \widetilde{K}.
\]
Then there exists a positive constant $\widetilde{K}_0$, that does not depend on $s$ and $T$, and there exists a constant $C(\widetilde{K}_0) >0$ such that, if $\widetilde{K}\leq \widetilde{K}_0$ and
$(U' ,\varphi'),\, (U'' ,\varphi'')\in H^{s+1}(\Omega_T )\times H^{s+1}(\partial\Omega_T)$, then
\[
\begin{array}{r}
\|\mathbb{L}''(\widehat{U},\widehat{{\Psi}})((U',{\Psi}'),(U'',{\Psi}''))\|_{H^s(\Omega_T)}\leq C(\widetilde{K}_0)\Bigl\{ {\nl(\widehat{{U}},\hat{f})\nr}_{s+1}{\nl({U}',\varphi')\nr}_{3}{\nl({U}'',\varphi'')\nr}_{3}\qquad\\[6pt]
+{\nl({U}',\varphi')\nr}_{s+1}{\nl({U}'',\varphi'')\nr}_{3}
+ {\nl({U}'',\varphi'')\nr}_{s+1}{\nl({U}',\varphi')\nr}_{3}\Bigr\},
\end{array}
\]
where ${\nl({U} ,\varphi )\nr}_{\ell}:=\|{U} \|_{H^{\ell}(\Omega_T)} +\| \varphi\|_{H^{\ell}(\partial\Omega_T)}$. 
If $({W}' ,\varphi'),\, ({W}'' ,\varphi'')\in H^{s}(\partial\Omega_T )\times H^{s+1}(\partial\Omega_T)$, then
\[
\begin{array}{l}
\|\mathbb{B}''(({W}',\varphi'),({W}'',\varphi''))\|_{H^s(\partial\Omega_T)}\leq C(\widetilde{K}_0)\Bigl\{
\|{W}'\|_{H^s(\partial\Omega_T)}\|\varphi''\|_{H^3(\partial\Omega_T)}\\[6pt]
\qquad+\|{W}'\|_{H^3(\partial\Omega_T)}\|\varphi''\|_{H^{s+1}(\partial\Omega_T)}+\|{W}''\|_{H^s(\partial\Omega_T)}\|\varphi'\|_{H^{3}(\partial\Omega_T)}+\|{W}''\|_{H^3(\partial\Omega_T)}\|\varphi'\|_{H^{s+1}(\partial\Omega_T)}\\[6pt]
\qquad\qquad+\|{W}'\|_{H^s(\partial\Omega_T)}
\|{W}''\|_{H^3(\partial\Omega_T)}+\|{W}'\|_{H^3(\partial\Omega_T)}
\|{W}''\|_{H^s(\partial\Omega_T)}\Bigr\}.
\end{array}
\]
\label{p2}
\end{proposition}

Without loss of generality we assume that the constant $\widetilde{K}_0=2C_*$, where $C_*$ is the constant from
(\ref{66}). By using (\ref{85}), (\ref{86}), and Proposition \ref{p2}, we obtain the following result.

\begin{lemma}
Let $\alpha \geq 4$. There exist $\delta >0$ sufficiently small, and $\theta_0 \geq 1$ sufficiently large, such that
for all $k=0,\ldots n-1$, and for all integer $s\in [3,\widetilde{\alpha}-1]$, we have the estimates
\begin{align}
\|{e}'_k\|_{H^s(\Omega_T)}\leq C\delta^2\theta_k^{L_1(s)-1}\Delta_k,\label{87}\\
\|\tilde{e}'_k\|_{H^s(\partial\Omega_T)}\leq C\delta^2\theta_k^{L_1(s)-1}\Delta_k,\label{88}
\end{align}
where $L_1(s)=\max \{ (s+1-\alpha )_+ +4-2\alpha ,s+2-2\alpha \}$.
\label{l4}
\end{lemma}

\begin{proof}
In view of (\ref{66}) (recall that $m =\tilde{\alpha} -4$), ($H_{n-1}$), and (\ref{80}), we estimate the ``coefficient'' of $\mathbb{L}''$ in (\ref{85}) as  follows:
\[
\sup_{\tau\in [0,1]}{\nl ({U}^a +{U}_k +\tau\delta{U}_k ,\varphi^a+\varphi_k+\tau\delta \varphi_k)\nr}_3\leq
C_* +\delta\theta_k^{(3-\alpha )_+} + \delta\theta_k^{2-\alpha }\Delta_k\leq C_* +C\delta \leq 2C_*
\]
for $\delta$ sufficiently small. Therefore, we may apply Proposition \ref{p2}:
\[
\|{e}'_k\|_{H^s(\Omega_T)}\leq C\Bigl( \delta^2\theta_k^{4 -2\alpha}\Delta_k^2\bigl(C_*+ {\nl({U}_k ,\varphi_k)\nr}_{s+1}+{\nl(\delta{U}_k ,\delta \varphi_k)\nr}_{s+1}\bigr)
+\delta^2\theta_k^{s+2-2\alpha}\Delta_k^2\Bigr)
\]
for $s\in [3,\widetilde{\alpha}-1]$. If $s+1 \neq \alpha $, it follows from (\ref{80}) that
\[
\|{e}'_k\|_{H^s(\Omega_T)}\leq C\delta^2\Delta_k^2\left\{ \theta_k^{(s+2-\alpha )_+ +12 -2\alpha} +
\theta_k^{s+7-2\alpha}\right\} \leq C\delta^2\theta_k^{L_1(s)-1}\Delta_k
\]
(here we have used the inequality $\theta_k\Delta_k\leq 1/2$). If $s+1=\alpha$ and $\alpha\geq 4$,
\[
\|{e}'_k\|_{H^s(\Omega_T)}\leq C\delta^2\Delta_k^2\left\{(C_* +\delta\log\theta_k +\delta\theta_k^{-1}\Delta_k)\theta_k^{4-2\alpha}  + \theta_k^{1-\alpha}\right\}
\leq C\delta^2\Delta_k^2\theta_k^{1-\alpha}\leq C\delta^2\theta_k^{L_1(\alpha -1 )-1}\Delta_k.
\]
Analogously, by using (\ref{86}), Proposition \ref{p2}, and the trace theorem, we get \eqref{88}.
\end{proof}

\paragraph{Estimate of the first substitution errors.} The first substitution errors can be rewritten as follows:
\begin{equation}
\begin{array}{rl}
{e}''_k &= {\cal L}'({U}_{k} ,{\Psi}_{k})(\delta{{U}}_k,\delta{\Psi}_{k})-
{\cal L}'(S_{\theta_k}{U}_{k} ,S_{\theta_k}{\Psi}_{k})(\delta{{U}}_k,\delta{\Psi}_{k})
\\[3pt]
 & {\displaystyle{=\int_0^1\mathbb{L}''\bigl({U}^a+S_{\theta_k}{U}_k +\tau (I-S_{\theta_k}){U}_k, {\Psi}^a+S_{\theta_k}{\Psi}_k} }
 \\[9pt]
& \quad + \tau (I -  S_{\theta_k}){\Psi}_k\bigr)\bigl((\delta{U}_k,\delta{\Psi}_k),
((I-S_{\theta_k}) {U}_k,(I-S_{\theta_k}) {\Psi}_k)\bigr) d\tau ,
\end{array}
\label{89}
\end{equation}
\begin{equation}
\begin{array}{rl}
\tilde{e}''_k &=
\bigl({\cal B}'({U}_{k},\varphi_{k})(\delta{{U}}_k,\delta \varphi_{k})-{\cal B}'(S_{\theta_k}{U}_{k},S_{\theta_k}\varphi_{k})(\delta{{U}}_k,\delta \varphi_{k})\bigr)|_{x_1=0}\\[6pt]
& =\mathbb{B}''\bigl(
(\delta{U}_k|_{x_1=0},\delta \varphi_k),(({U}_k-S_{\theta_k}{U}_k)|_{x_1=0},\varphi_k- S_{\theta_k}\varphi_k)\bigr).
\end{array}\label{90}
\end{equation}

\begin{lemma}
Let $\alpha \geq 4$. There exist $\delta >0$ sufficiently small, and $\theta_0 \geq 1$ sufficiently large, such that
for all $k=0,\ldots n-1$, and for all integer $s\in [6,\widetilde{\alpha}-2]$, one has
\begin{align}
\|{e}''_k\|_{H^s(\Omega_T)}\leq C\delta^2\theta_k^{L_2(s)-1}\Delta_k,\label{91}\\
\|\tilde{e}''_k\|_{H^s(\partial\Omega_T)}\leq C\delta^2\theta_k^{L_2(s)-1}\Delta_k,\label{92}
\end{align}
where $L_2(s)=\max \{ (s+1-\alpha )_+ +6-2\alpha ,s+5-2\alpha \}$.
\label{l5}
\end{lemma}

\begin{proof}
It follows from (\ref{66}), ($H_{n-1}$), (\ref{82}), and (\ref{83}) that
\[
\sup_{\tau\in [0,1]} {\nl (U^a+S_{\theta_k}U_k +\tau (I-S_{\theta_k})U_k ,\varphi^a+S_{\theta_k}\varphi_k+\tau (I-S_{\theta_k}) \varphi_k)\nr}_3\leq 2C_*
\]
for $\delta$ sufficiently small, i.e., we may apply Proposition \ref{p2} for estimating $\mathbb{L}''$ in (\ref{89}). Using again (\ref{66}), ($H_{n-1}$), (\ref{82}), and (\ref{83}), for $s+1\neq {\alpha}$ and $s+1\leq \tilde{\alpha}$ 
we get
\[
\|{e}''_k\|_{H^s(\Omega_T)}\leq C\Bigl\{ \delta^2\theta_k^{5 -2\alpha}\Delta_k\bigl(C_*+ \delta\theta_k^{(s+1-\alpha )_+}+\delta\theta_k^{s+1-\alpha }\bigr)+\delta^2\theta_k^{s+3-2\alpha}\Delta_k\Bigr\}
\leq  C\delta^2\theta_k^{L_2(s)-1}\Delta_k.
\]
Similarly, but exploiting (\ref{84}) instead of (\ref{83}), for the case $s+1= {\alpha}$ we obtain
\[
\begin{array}{c}
\|{e}''_k\|_{H^s(\Omega_T)}\leq C\Bigl\{ \delta^2\theta_k^{5 -2\alpha}\Delta_k(C_*+ \delta\log\theta_k+\delta )+\delta^2\theta_k^{2 -\alpha}\Delta_k\Bigr\}
\\[6pt]
\qquad\qquad 
\leq C\delta^2\Delta_k\left\{ \theta_k^{6-2\alpha }+\theta_k^{2-\alpha}\right\}\leq C\delta^2\theta_k^{L_2(\alpha -1 )-1}\Delta_k
\end{array}
\]
for $\alpha \geq 4$.

By virtue of (\ref{90}), the trace theorem, and Proposition \ref{p2}, we have
\[
\begin{array}{l}
\|\tilde{e}''_k\|_{H^s(\Omega_T)}\leq C\Bigl\{ [\delta{U}_k ]_{s+1,*,T}\|(1-S_{\theta_k}) \varphi_k\|_{H^3(\partial\Omega_T)}+
\|\delta{U}_k \|_{H^3(\Omega_T)}\| (1-S_{\theta_k}) \varphi_k\|_{H^{s+1}(\partial\Omega_T)}\\[6pt]
\quad 
+\|(I-S_{\theta_k}){U}_k ]_{H^{s+1}(\Omega_T)}\| \delta \varphi_k\|_{H^{3}(\partial\Omega_T)}
+\|(I-S_{\theta_k}){U}_k \|_{H^{3}(\Omega_T)}\| \delta \varphi_k\|_{H^{s+1}(\partial\Omega_T)}\\[6pt]
\qquad 
+\|\delta{U}_k \|_{H^{s+1}(\Omega_T)}\|(I-S_{\theta_k}){U}_k \|_{H^{3}(\Omega_T)}+\|\delta{U}_k \|_{H^{3}(\Omega_T)}\|(I-S_{\theta_k}){U}_k \|_{H^{s+1}(\Omega_T)}
\Bigr\}.
\end{array}
\]
Then, ($H_{n-1}$) and (\ref{82}) imply (\ref{92}). 
\end{proof}

\paragraph{Construction and estimate of the modified state.} Since the approximate solution satisfies the strict inequalities (\ref{6}) (for all $x\in\overline{\Omega_T}$) and (\ref{16'})  (see Lemma \ref{l2}) and since we shall require that the smooth modified state vanishes in the past, the state $(U^a+U_{n+1/2},\varphi^a+\varphi_{n+1/2})$ will satisfy (\ref{6}) and (\ref{16'}) for a sufficiently short time $T>0$. Therefore, while constructing the modified state we may focus only on constraint (\ref{20}), i.e., the first boundary condition in (\ref{14}).

\begin{proposition}
Let $\alpha \geq 4$. The exist some functions $U_{n+1/2}$ and $\varphi_{n+1/2}$, that vanish in the past, and such that
$(U^a+U_{n+1/2},\varphi^a+\varphi_{n+1/2})$ satisfies (\ref{20}), and inequalities (\ref{6}) and (\ref{16'}) for a sufficiently short time $T$. Moreover, these functions satisfy
\begin{equation}
\varphi_{n+1/2}=S_{\theta_n}\varphi_n, \quad
p_{n+1/2}=S_{\theta_n}p_n,\quad v_{j,n+1/2}=S_{\theta_n}v_{j,n}\quad  (j=2,3),\quad
S_{n+1/2}=S_{\theta_n}S_n,  \label{93}
\end{equation}
and
\begin{equation}
\|U_{n+1/2} - S_{\theta_n}U_n\|_{H^s(\Omega_T)}\leq C\delta\theta_n^{s+1-\alpha}\quad \mbox{for}\ s\in [3,\tilde{\alpha}+3].
\label{94}
\end{equation}
for sufficiently small $\delta>0$ and $T>0$, and a sufficiently large $\theta_0\geq 1$.
\label{p3}
\end{proposition}

\begin{proof}
Actually, estimate \eqref{94} which we are going to prove hold for every $s\geq 3$ but below we will need it only for $s\in [3,\tilde{\alpha}+3]$. Let $\varphi_{n+1/2}$, the pressure $p_{n+1/2}$, the entropy $S_{n+1/2}$, and the tangential components of the velocity $v_{n+1/2}$ are defined by (\ref{93}). We define $v_{1,n+1/2}$ as in \cite{Tr}:
\[
v_{1,n+1/2}:= S_{\theta_n}v_{1,n} +{\cal R}_T{\cal G},
\]
where
\[
{\cal G}= \partial_t \varphi_{n+1/2} -(S_{\theta_n}v_{1,n})|_{x_1=0}
+\sum_{j=2}^3\bigl(
(v_j^{a}+v_{j,n+1/2} )\partial_j\varphi_{n+1/2} + v_{j,n+1/2}\partial_j\varphi^a\bigr)\bigr|_{x_1=0},
\]
and ${\cal R}_T:\; H^s(\partial\Omega_T)\longrightarrow H^{s+1}(\Omega_T)$ is the lifting operator from the boundary to the interior. To get the estimate of $v_{1,n+1/2} - S_{\theta_n}v_{1,n}$ we use the following decompositions:
\[
\begin{array}{r}
{\cal G} =S_{\theta_n}{\cal B}_{1}({U}_n|_{x_1=0},\varphi_n) - \partial_t(1-S_{\theta_n})\varphi_n + (1-S_{\theta_n})\partial_t\varphi_n {\displaystyle +
\sum_{j=2}^3\bigl( (v_j^{a}+S_{\theta_n}v_{j,n} )\partial_jS_{\theta_n}\varphi_{n} }
\\[3pt]
-S_{\theta_n}((v_j^{a}+v_{j,n})
\partial_j\varphi_n)
+ (S_{\theta_n}v_{j,n})\partial_j\varphi^a-S_{\theta_n}(v_{j,n}\partial_j\varphi^a)\bigr)\bigr|_{x_1=0}
\end{array}
\]
and
\[
\begin{array}{l}
{\cal B}_{1}({U}_n|_{x_1=0},\varphi_n)=\mathbb{B}_{v}({U}_{n-1}|_{x_1=0},\varphi_{n-1})+\partial_t(\delta \varphi_{n-1})
\\[3pt]
\qquad{\displaystyle
+\sum_{j=2}^3\bigl( (v_j^{a}+v_{j,n-1} )\partial_j(\delta \varphi_{n-1}) +\delta v_{j,n-1}
\partial_j(\varphi^a+\varphi_n) -\delta v_{1,n-1}\bigr)\bigr|_{x_1=0}\,,}
\end{array}
\]
where ${\cal B}_{1}$ denotes the first row of the boundary operator ${\cal B}$ in (\ref{70}).

Exploiting point $c)$ of ($H_{n-1}$), one has
\[
\|{\cal R}_T(S_{\theta_n}{\cal B}_1({U}_{n-1}|_{x_1=0},\varphi_{n-1}))\|_{H^s(\Omega_T)}\leq C\| S_{\theta_n}{\cal B}_1({U}_{n-1}|_{x_1=0},\varphi_{n-1})\|_{H^s(\partial\Omega_T)}
\]
\[
\leq \left\{ \begin{array}{ll} C\theta_n^{s-\alpha }\|{\cal B}_1({U}_{n-1}|_{x_1=0},\varphi_{n-1})\|_{H^{\alpha}(\partial\Omega_T)}&\quad \mbox{for}\ s\in [\alpha,\tilde{\alpha}+3],\\[3pt]
C\|{\cal B}_1({U}_{n-1}|_{x_1=0},\varphi_{n-1})\|_{H^{s+1}(\partial\Omega_T)}&\quad \mbox{for}\ s\in [3,\alpha -1]
\end{array}\right.
\leq C\delta\theta_n^{s-\alpha}\quad \mbox{for}\ s\in [3,\tilde{\alpha}+3].
\]
Using (\ref{72}) and point $a)$ of ($H_{n-1}$), we get
\[
\begin{array}{l}
\|{\cal R}_T(S_{\theta_n}\partial_t(\delta \varphi_{n-1}))\|_{H^s(\Omega_T)}\leq C\| S_{\theta_n}\partial_t(\delta \varphi_{n-1})\|_{H^s(\partial\Omega_T)}
\\[6pt]
\qquad\leq C\theta_n^{s-2}\|\delta \varphi_{n-1}\|_{H^3(\partial\Omega_T)}\leq C\theta_n^{s-2}\delta\theta_{n-1}^{2-\alpha }
\theta_{n-1}^{-1}\leq C\delta\theta_n^{s-\alpha -1}
\end{array}
\]
for $s\in [3,\tilde{\alpha} +3]$. We also obtain
\[
\begin{array}{l}
\|{\cal R}_T\bigl(S_{\theta_n}((v_j^{a}+v_{j,n-1} )|_{x_1=0}\,\partial_j(\delta \varphi_{n-1})) 
\bigr)\|_{H^s(\Omega_T)}\leq C\theta_n^{s-3}\|(v_j^{a}+v_{j,n-1} )|_{x_1=0}\,\partial_j(\delta \varphi_{n-1})\|_{H^3(\partial\Omega_T)}
\\[6pt]
\qquad\leq C\theta_n^{s-3}\Bigl\{\|\delta \varphi_{n-1}\|_{H^4(\partial\Omega_T)}\|{U}^a +{U}_{n-1}
\|_{H^3(\Omega_T)}
\\[6pt]
\qquad\quad +
\|\delta \varphi_{n-1}\|_{H^3(\partial\Omega_T)}\|{U}^a +{U}_{n-1}\|_{H^7(\Omega_T)}\Bigr\}\leq C\theta_n^{s-3}\delta\theta_n^{2-\alpha}C_*
\leq C\delta\theta_n^{s-\alpha -1}
\end{array}
\]
for $j=2,3$ and $s\in [3,\tilde{\alpha} +3]$. Estimating similarly the remaining terms containing
in ${\cal R}_T(S_{\theta_n}{\cal B}_1({U}_{n}|_{x_1=0},\varphi_{n}))$, we finally obtain 
\[
\|{\cal R}_T(S_{\theta_n}{\cal B}_1({U}_{n}|_{x_1=0},\varphi_{n}))\|_{H^s(\Omega_T)}\leq C\delta\theta_n^{s-\alpha },\quad
s\in [3,\tilde{\alpha} +3].
\]

We now need to derive estimates for the remaining terms containing in ${\cal R}_T{\cal G}$. For $s\in [\alpha , \tilde{\alpha} +3]$ one has
\[
\begin{array}{r}
\|{\cal R}_T(-\partial_t(1-S_{\theta_n})\varphi_n +  (1-S_{\theta_n})\partial_t\varphi_n)\|_{H^s(\Omega_T)}
\leq C \bigl\{
\|\partial_t(S_{\theta_n}\varphi_n) \|_{H^s(\partial\Omega_T)}+\|S_{\theta_n}(\partial_t\varphi_n) \|_{H^s(\partial\Omega_T)} \bigr\}\quad\\[6pt]
\leq C\bigl\{
\|S_{\theta_n}\varphi_n \|_{H^{s+1}(\partial\Omega_T)}
+\theta_n^{s-\alpha}\| \varphi_n\|_{H^{\alpha+1}(\partial\Omega_T)}\bigr\}\leq C\delta\theta_n^{s+1-\alpha},
\end{array}
\]
while for $s\in [3 , \tilde{\alpha} -1]$ we obtain (recall that $\tilde{\alpha}=\alpha +3$)
\[
\|{\cal R}_T(\partial_t(1-S_{\theta_n})\varphi_n )\bigr\|_{H^s(\Omega_T)}\leq C\delta\theta_n^{s+1-\alpha},
\]
\[
\|{\cal R}_T((1-S_{\theta_n})\partial_t\varphi_n )\|_{H^s(\Omega_T)}\leq C\theta_n^{s-\alpha}\| \varphi_n\|_{H^{\alpha+1}(\partial\Omega_T)}\leq C\delta\theta_n^{s+1-\alpha}.
\]
Here we have, in particular, used Lemma \ref{l3}. We do not get estimates for all the remaining terms containing in ${\cal R}_T{\cal G}$ and leave corresponding calculations to the reader.  Collecting these estimates and the estimates above, we finally have
\[
\|v_{1,n+1/2} - S_{\theta_n}v_{1,n}\|_{H^s(\Omega_T)} \leq C\delta\theta_n^{s+1-\alpha},\quad
s\in [3,\tilde{\alpha} +3],
\]
that is equivalent to \eqref{94}.
\end{proof}

\paragraph{Estimate of the second substitution errors.} The second substitution errors
\[
{e}'''_k = {\cal L}'(S_{\theta_k}{U}_{k} ,S_{\theta_k}{\Psi}_{k})(\delta{{U}}_k,\delta{\Psi}_{k})- {\cal L}'({U}_{k+1/2} ,{\Psi}_{k+1/2})(\delta{{U}}_k,\delta{\Psi}_{k})
\]
and
\[
\tilde{e}'''_k =
\bigl({\cal B}'(S_{\theta_k}{U}_{k},S_{\theta_k}\varphi_{k})(\delta{{U}}_k,\delta \varphi_{k})-
{\cal B}'({U}_{k+1/2},\varphi_{k+1/2})(\delta{{U}}_k,\delta \varphi_{k})
\bigr)|_{x_1=0}
\]
can be written as
\begin{equation}
{\displaystyle
{e}'''_k =\int_0^1\mathbb{L}''\bigl({U}^a+{U}_{k+1/2} +\tau (S_{\theta_k}{U}_k-{U}_{k+1/2}),  }
{\Psi}^a+S_{\theta_k}{\Psi}_k) \bigl((\delta{U}_k,\delta{\Psi}_k),
(S_{\theta_k} {U}_k -{U}_{k+1/2},0) \bigr) d\tau ,
\label{95}
\end{equation}
\begin{equation}
\tilde{e}'''_k
 =\mathbb{B}''\bigl(
(\delta{U}_k|_{x_1=0},\delta \varphi_k),((S_{\theta_k}{U}_k-{U}_{k+1/2})|_{x_1=0},0)\bigr).
\label{96}
\end{equation}
Employing (\ref{95}) and (\ref{96}), we get the following result.

\begin{lemma}
Let $\alpha \geq 4$. There exist $\delta >0$, $T>0$ sufficiently small, and $\theta_0 \geq 1$ sufficiently large, such that for all $k=0,\ldots n-1$, and for all integer $s\in [3,\widetilde{\alpha}-1]$, one has
\begin{equation}
\|{e}'''_k\|_{H^s(\Omega_T)}\leq C\delta^2\theta_k^{L_3(s)-1}\Delta_k\label{97}
\end{equation}
and $\tilde{e}'''_k=0$,
where $L_3(s)=\max \{ (s+1-\alpha )_+ +8-2\alpha ,s+5-2\alpha \}$.
\label{l6}
\end{lemma}

\begin{proof}
Using Lemma \ref{l3} and Proposition \ref{p3}, we obtain the estimate
\[
\sup_{\tau\in [0,1]}{\nl ({U}^a+{U}_{k+1/2} +\tau (S_{\theta_k}{U}_k-{U}_{k+1/2}),
\varphi^a+S_{\theta_k}\varphi_k)\nr}_{3}\leq 2C_*
\]
for $\delta$ sufficiently small, i.e., we may apply Proposition \ref{p2}. Similarly, one gets
\[
\begin{array}{l}
{\nl ({U}^a+{U}_{k+1/2} +\tau (S_{\theta_k}{U}_k-{U}_{k+1/2}),
\varphi^a+S_{\theta_k}\varphi_k)\nr}_{s+1}
\\[6pt]\qquad
\leq C\bigl\{C_*+\delta\theta_k^{s+2-\alpha} +\delta\theta_k^{(s+1-\alpha)_++1}\bigr\}\leq  C\delta\theta_n^{(s+1-\alpha)_++1}.
\end{array}
\]
Applying Proposition \ref{p2}, we obtain (\ref{97}):
\[
\begin{array}{rl}
\|{e}'''_k\|_{H^s(\Omega_T)}\leq & C\Bigl\{ \delta\theta_k^{(s+1-\alpha)_++1}\delta\theta_k^{2-\alpha}\Delta_k\delta\theta_k^{4-\alpha}+
\delta\theta_k^{s+-\alpha}\Delta_k\delta\theta_k^{4-\alpha}\\[6pt]
& \quad +\delta\theta_k^{2-\alpha}\Delta_k\delta\theta_k^{s+2-\alpha}\Bigr\}\leq C\delta^2\theta_k^{L_3(s)-1}\Delta_k.
\end{array}
\]
Using the explicit form of $\mathbb{B}''$, we easily get $\tilde{e}'''_k=0$.
\end{proof}

\paragraph{Estimate of the last error term.} We now estimate the last error term 
\[
{ D}_{k+1/2}\delta\Psi_k=\frac{\delta\Psi_k}{\partial_1(\Phi^a+\Psi_{n+1/2})}\,{ R}_k,
\]
where ${ R}_k:=\partial_1\left\{ \mathbb{L}({U}^a +{U}_{k+1/2} ,{\Psi}^a+{\Psi}_{k+1/2})\right\}$. Note that
\[
|\partial_1(\Phi^a+\Psi_{n+1/2})|=|1+\partial_1(\Psi^a+\Psi_{n+1/2})|\geq 1/2,
\]
provided that $T$ and $\delta$ are small enough.

\begin{lemma}
Let $\alpha \geq 5$. There exist $\delta >0$, $T>0$ sufficiently small, and $\theta_0 \geq 1$ sufficiently large, such that for all $k=0,\ldots n-1$, and for all integer $s\in [3,\widetilde{\alpha}-2]$, one has
\begin{equation}
\|{D}_{k+1/2}\delta{\Psi}_k\|_{H^s(\Omega_T)}\leq C\delta^2\theta_k^{L(s)-1}\Delta_k,\label{98}
\end{equation}
where $L(s)=\max \{ (s+2-\alpha )_+ +8-2\alpha ,(s+1-\alpha)_++9-2\alpha ,s+6-2\alpha \}$.
\label{l7}
\end{lemma}

\begin{proof} The proof follows from the arguments as in \cite{Al,CS} (see also \cite{Tr}).
Using the Moser-type and embedding inequalities, we obtain
\begin{equation}
\begin{array}{l}
\|[{ D}_{k+1/2}\delta\Psi_k\|_{H^s(\Omega_T)}\leq C\Bigl\{\|\delta \varphi_k\|_{H^s(\partial\Omega_T)}\|{ R_k}\|_{H^3(\Omega_T)}\\[6pt]
\qquad +\|\delta \varphi_k\|_{H^3(\partial\Omega_T)}\bigl(\|{ R_k}\|_{H^s(\Omega_T)} +\|{ R_k}\|_{H^3(\Omega_T)}\|\varphi^a+\varphi_{k+1/2}\|_{H^s(\partial\Omega_T)} \bigr)\Bigr\}
\end{array}\label{99}
\end{equation}
(note that $\|\partial_1(\Psi^a+\Psi_{n+1/2})\|_{H^s(\Omega_T)}\leq C\|\varphi^a+\varphi_{k+1/2}\|_{H^s(\partial\Omega_T)}$).
To estimate ${ R}_k$ we utilize the decomposition
\[
\begin{array}{r}
\mathbb{L}({U}^a +{U}_{k+1/2} ,{\Psi}^a+{\Psi}_{k+1/2})={\cal L}({U}_{k} ,{\Psi}_{k})-{f}^a 
 +\mathbb{L}({U}^a +{U}_{k+1/2} ,{\Psi}^a+{\Psi}_{k+1/2})
\\[6pt]
 - \mathbb{L}({U}^a +{U}_{k} ,{\Psi}^a+{\Psi}_{k})
={\cal L}({U}_{k} ,{\Psi}_{k})-{f}^a
{\displaystyle +\int_0^1\mathbb{L}'\bigl({U}^a+{U}_{k} +\tau ({U}_{k+1/2}-{U}_{k}),  }
\\[12pt]
{\Psi}^a+{\Psi}_k +\tau({\Psi}_{k+1/2}-{\Psi}_k)\bigr) ({U}_{k+1/2}-{U}_{k},{\Psi}_{k+1/2}-{\Psi}_k) d\tau .
\end{array}
\]

Clearly,
\begin{equation}
\|{ R}\|_{H^s(\Omega_T)}\leq \|{\cal L}({U}_{k} ,{\Psi}_{k})-{f}^a\|_{H^s(\Omega_T)}+\sup_{\tau\in [0,1]}\|\mathbb{L}'(\ldots ) (\ldots )\|_{H^{s+1}(\Omega_T)}\label{100}
\end{equation}
(for short we drop the arguments of $\mathbb{L}'$). It follows from point $b)$ of ($H_{n-1}$) that
\begin{equation}
\|{\cal L}({U}_{k} ,{\Psi}_{k})-{f}^a\|_{H^{s+1}(\Omega_T)}\leq 2\delta\theta_k^{s-\alpha}
\label{101}
\end{equation}
for $s\in [3,\tilde{\alpha}-3]$.
We estimate $\mathbb{L}'$ similarly to $\mathbb{L}''$ (see Proposition \ref{p2}). One has
\[
\sup_{\tau\in [0,1]}{\nl ({U}^a+{U}_{k} +\tau ({U}_{k+1/2}-{U}_{k}),
\varphi^a+\varphi_k +\tau(\varphi_{k+1/2}-\varphi_k))\nr}_{3}\leq 2C_*
\]
for $\delta$ small enough. Then, omitting detailed calculations, we get the estimate
\[
\|\mathbb{L}'(\ldots ) (\ldots )\|_{H^{s+1}(\Omega_T)}\leq C\delta (\theta_k^{s+3-\alpha}+\theta_k^{(s+2-\alpha )_++5-\alpha})
\]
for $s\in [3,\tilde{\alpha}-3]$. This estimate, (\ref{100}), and  (\ref{101}) imply
\begin{equation}
\|{ R}\|_{H^{s}(\Omega_T)}\leq C\delta (\theta_k^{s+3-\alpha}+\theta_k^{(s+2-\alpha )_++5-\alpha})\label{102}
\end{equation}
for $s\in [3,\tilde{\alpha}-3]$. For $s=\tilde{\alpha}-2$ we estimate as follows:
\[
\begin{array}{l}
\|{ R}\|_{H^{s}(\Omega_T)}\leq \|\mathbb{L}({U}^a +{U}_{k+1/2} ,{\Psi}^a+{\Psi}_{k+1/2})\|_{H^{s+1}(\Omega_T)}\\[6pt]
\qquad\leq C {\nl ({U}^a+({U}_{k+1/2}-S_{\theta_n}{U}_{k})+S_{\theta_n}{U}_{k},
\varphi^a +S_{\theta_n}\varphi_{k})\nr}_{s+2}\leq C\delta\theta_k^{s+3-\alpha}.
\end{array}
\]
That is, we get estimate (\ref{102}) for $s\in [3,\tilde{\alpha}-2]$. Using then (\ref{99}),
we obtain (\ref{98}), provided that $\alpha \geq 5$. 
\end{proof}

\paragraph{Convergence of the iteration scheme.}
Lemmas \ref{l4}--\ref{l7} yield the estimate of ${e}_n$ and $\tilde{e}_n$ defined in (\ref{77}) as the sum of all the errors of the $k$th step.

\begin{lemma}
Let $\alpha \geq 5$. There exist $\delta >0$, $T>0$ sufficiently small, and $\theta_0 \geq 1$ sufficiently large, such that
for all $k=0,\ldots n-1$, and for all integer $s\in [3,\widetilde{\alpha}-2]$, one has
\begin{equation}
\|{e}_k\|_{H^s(\Omega_T)}+\|\tilde{e}_k\|_{H^s(\partial\Omega_T)}\leq C\delta^2\theta_k^{L(s)-1}\Delta_k,
\label{103}
\end{equation}
where $L(s)$ is defined in Lemma \ref{l7}.
\label{l8}
\end{lemma}

\begin{remark}{\rm
In principle, we could try to use the advantage of the fact that in the tame estimate \eqref{38} we do not lose derivatives from the source term $f$ to the solution. To this end, in Lemma \ref{l8} we could estimate errors ${e}_n$ and $\tilde{e}_n$ separately. However, this does not reduce the number of derivatives lost from the initial data to the solution in the existence Theorem \ref{t2}. In fact, we can even use a roughened version of estimate \eqref{38}
in which we lose one derivative from $f$ to the solution.}\label{rough}
\end{remark}

Lemma \ref{l8} gives the estimate of the accumulated errors ${E}_n$ and $\widetilde{E}_n$.

\begin{lemma}
Let $\alpha \geq 7$.  There exist $\delta >0$, $T>0$ sufficiently small, and $\theta_0 \geq 1$ sufficiently large, such that
\begin{equation}
\|{E}_n\|_{H^{\alpha +2}(\Omega_T)}+\|\widetilde{E}_n\|_{H^{\alpha +2}(\partial\Omega_T)}\leq C\delta^2\theta_n,
\label{104}
\end{equation}
where $L(s)$ is defined in Lemma \ref{l7}.
\label{l9}
\end{lemma}

\begin{proof} One can check that $L(\alpha +2)\leq 1$ if $\alpha \geq 7$. It follows from (\ref{103}) that
\[
{\nl ({E}_n,\widetilde{E}_n)\nr}_{\alpha +2}\leq \sum_{k=0}^{n-1}{\nl ({e}_k,\widetilde{e}_k)\nr}_{\alpha +2}\leq \sum_{k=0}^{n-1}C\delta^2\Delta_k\leq C\delta^2\theta_n
\]
for $\alpha \geq 7$ and $\alpha +2\in [3,\tilde{\alpha}-2]$, i.e., $\tilde{\alpha}\geq \alpha +4$. The minimal possible
$\tilde{\alpha}$ is $\alpha +4$, i.e., our choice $\tilde{\alpha}= \alpha +4$ is suitable.
\end{proof}

We now derive the estimates of the source terms ${f}_n$ and ${g}_n$ defined in (\ref{79}).

\begin{lemma}
Let $\alpha \geq 7$.  There exist $\delta >0$, $T>0$ sufficiently small, and $\theta_0 \geq 1$ sufficiently large, such that for all integer $s\in [3,\widetilde{\alpha}+1]$, one has
\begin{align}
& \|{f}_n\|_{H^{s}(\Omega_T)}\leq  C\Delta_n\bigl\{ \theta_n^{s-\alpha -2}\left( \|{f}^a\|_{H^{\alpha +1}(\Omega_T)}+\delta^2\right)+\delta^2\theta_n^{L(s)-1} \bigr\},\label{105}\\[6pt]
 & \|{g}_n\|_{H^{s}(\partial\Omega_T)}\leq  C\delta^2\Delta_n\bigl( \theta_n^{L(s)-1}+\theta_n^{s-\alpha -2}\bigr).
\label{106}
\end{align}
\label{l10}
\end{lemma}

\begin{proof} It follows from (\ref{79}) that
\[
{f}_n =(S_{\theta_n}-S_{\theta_{n-1}}){f}^a-(S_{\theta_n}-S_{\theta_{n-1}}){E}_{n-1}-
S_{\theta_n}{e}_{n-1}.
\]
Using (\ref{72}), (\ref{74}), (\ref{103}), and (\ref{104}), we obtain the estimates
\[
\|(S_{\theta_n}-S_{\theta_{n-1}}){f}^a\|_{H^{s}(\Omega_T)}\leq C\theta_{n-1}^{s-\alpha -2}
\|{f}^a\|_{H^{\alpha +1}(\Omega_T)}\Delta_{n-1},
\]
\[
\|(S_{\theta_n}-S_{\theta_{n-1}}){E}_{n-1}\||_{H^{s}(\Omega_T)}\leq C\theta_{n-1}^{s-\alpha -3}\|{E}_{n-1}
\|_{H^{\alpha +2}(\Omega_T)}\Delta_{n-1}\leq C\delta^2\theta_{n-1}^{s-\alpha -2}\Delta_{n-1},
\]
\[
\|S_{\theta_n}{e}_{n-1}\|_{H^{s}(\Omega_T)}\leq C\delta^2\theta_n^{L(s)-1}\Delta_{n-1}.
\]
Using the inequalities $\theta_{n-1}\leq \theta_n\leq \sqrt{2}\theta_{n-1}$, $\theta_{n-1}\leq 3\theta_n$,
and $\Delta_{n-1}\leq 3\Delta_n$, from the above estimates we deduce (\ref{105}). Similarly, we get (\ref{106}).
\end{proof}

We are now in a position to obtain the estimate of the solution to problem (\ref{75}) by employing the tame estimate
(\ref{38}). Then the estimate of $(\delta U_n,\delta\varphi_n)$ follows from formula (\ref{76}).

\begin{lemma}
Let $\alpha \geq 7$.  There exist $\delta >0$, $T>0$ sufficiently small, and $\theta_0 \geq 1$ sufficiently large, such that for all integer $s\in [3,\widetilde{\alpha}]$, one has
\begin{equation}
\|\delta U_n\|_{H^{s}(\Omega_T)}+\|\delta \varphi_n\|_{H^{s}(\partial\Omega_T)}\leq  \delta\theta_n^{s-\alpha -1}\Delta_n.
\label{107}
\end{equation}
\label{l11}
\end{lemma}

\begin{proof}
Without loss of generality we can take the constant $K_0$ appearing in estimate (\ref{38}) that $K_0=2C_*$, where
$C_*$ is the constant from (\ref{66}). In order to apply Theorem \ref{t4}, by using (\ref{83}) and (\ref{94}), we check that
\[
\|{U}^a+{U}_{n+1/2}\|_{H^{6}(\Omega_T)}+\|\varphi^a+S_{\theta_n}\varphi_n\|_{H^{6}(\partial\Omega_T)}\leq 2C_*
\]
for $\alpha \geq 7$ and $\delta$ small enough. That is, assumption (\ref{37}) is satisfied for the coefficients of problem (\ref{75}). By applying the tame estimate (\ref{38}), for $T$ small enough one has
\begin{equation}
\begin{array}{l}
{\displaystyle
\|\delta\dot{{U}}_n\|_{H^s(\Omega_T)}+\|\delta \varphi_n\|_{H^{s}(\partial\Omega_T)}\leq C\Bigl\{
\|{f}_n\|_{H^{s}(\Omega_T)}+ \|{g}_n\|_{H^{s+1}(\partial\Omega_T)} }\\[9pt]
\qquad +\bigl( \|{f}_n\|_{H^{3}(\Omega_T)}+ \|{g}_n\|_{H^{4}(\partial\Omega_T)} \bigr)\bigl(
\|{U}^a +{U}_{n+1/2}\|_{H^{s+3}(\Omega_T)}+\|\varphi^a +S_{\theta_n}\varphi_n\|_{H^{s+3}(\partial\Omega_T)}\bigr)\Bigr\}.
\end{array}
\label{108}
\end{equation}

Using Moser-type inequalities, from formula (\ref{76}) we obtain
\[
\|\delta{U}_n\|_{H^s(\Omega_T)}\leq  \|\delta\dot{{U}}_n\|_{H^s(\Omega_T)} +C\bigl\{\|\delta \varphi_n\|_{H^{s}(\partial\Omega_T)}
+ \|\delta \varphi_n\|_{H^{3}(\partial\Omega_T)}\|\varphi^a +S_{\theta_n}\varphi_n\|_{H^{s}(\partial\Omega_T)}
\bigr\}.
\]
Then (\ref{108}) yields
\begin{equation}
\begin{array}{l}
{\displaystyle
\|\delta{{U}}_n\|_{H^s(\Omega_T)}+\|\delta \varphi_n\|_{H^{s}(\partial\Omega_T)}\leq C\Bigl\{
\|{f}_n\|_{H^{s}(\Omega_T)}+ \|{g}_n\|_{H^{s+1}(\partial\Omega_T)} }\\[9pt]
\qquad +\bigl( \|{f}_n\|_{H^{3}(\Omega_T)}+ \|{g}_n\|_{H^{4}(\partial\Omega_T)} \bigr)\bigl(
\|{U}^a +{U}_{n+1/2}\|_{H^{s+3}(\Omega_T)}+\|\varphi^a +S_{\theta_n}\varphi_n\|_{H^{s+3}(\partial\Omega_T)}\bigr)\Bigr\}
\end{array}
\label{109}
\end{equation}
for all integer $s\in [6,\widetilde{\alpha}]$. Below we can actually use a roughened version of \eqref{109} (see Remark
\ref{rough}). Applying Lemma \ref{l11}, (\ref{83}), and Proposition \ref{p3},
from (\ref{109}) we derive the estimate
\begin{equation}
\begin{array}{l}
\|\delta{{U}}_n\|_{H^s(\Omega_T)}+\|\delta \varphi_n\|_{H^{s}(\partial\Omega_T)}\leq C
\bigl\{ \theta_n^{s-\alpha -1}\left( \|{f}^a\|_{H^{\alpha+1}(\Omega_T)}+\delta^2\right)+\delta^2\theta_n^{L(s+1)-1} \bigr\}\Delta_n\\[9pt]
\qquad  +C\delta\Delta_n\bigr\{\theta_n^{2-\alpha}\left( \|{f}^a
\|_{H^{\alpha+1}(\Omega_T)}
+\delta^2\right)
+\delta^2\theta_n^{9-2\alpha}\bigr\}\bigl\{ C_*+\theta_n^{(s+3-\alpha )_+}+\theta_n^{s+4-\alpha}\bigr\}.
\end{array}
\label{110}
\end{equation}
Exactly as in \cite{CS}, we can check that
the inequalities
\begin{equation}
\begin{array}{l}
L(s+1)\leq s-\alpha,\quad (s +3-\alpha)_+ +2-\alpha \leq s-\alpha -1,\\[3pt]
(s +3-\alpha)_+ +9-2\alpha \leq s-\alpha -1,\\[3pt]
s+6-2\alpha \leq s-\alpha -1,\quad s+13-3\alpha \leq s-\alpha -1
\end{array}
\label{111}
\end{equation}
hold for $\alpha\geq 7$ and $s\in [3,\tilde{\alpha}]$.
Thus, (\ref{110}) and (\ref{68}) yield
\[
\|\delta{{U}}_n\|_{H^s(\Omega_T)}+\|\delta \varphi_n\|_{H^{s}(\partial\Omega_T)}\leq C\left( \delta_0(T)+\delta^2\right)
\theta_n^{s-\alpha -1}\Delta_n \leq \delta\theta_n^{s-\alpha -1}\Delta_n
\]
for $\delta$ and $T$ small enough. 
\end{proof}

\begin{remark}{\rm As we can see, Lemma \ref{l11} with $\tilde{\alpha}=\alpha+4$ is absolutely analogous to Lemma 16 in \cite{CS}. In this sense,
the ``gain of one derivative for the front'' in the tame estimate gives no advantage in the realization of the Nash-Moser method. This is caused by the fact that even if in point a) of ($H_{n-1}$) we had the $H^{s+1}$--norm
of $\delta\varphi_k$ we could never use this advantage before the proof of Lemma \ref{l11}.
}
\end{remark}

Inequality (\ref{107}) is point $a)$ of ($H_n$). It remains to prove points $b)$ and $c)$ of ($H_n$).

\begin{lemma}
Let $\alpha \geq 7$.  There exist $\delta >0$, $T>0$ sufficiently small, and $\theta_0 \geq 1$ sufficiently large, such that for all integer $s\in [3,\widetilde{\alpha}-2]$
\begin{equation}
\|{\cal L}({U}_n,{\Psi}_n)-{f}^a\|_{H^s(\Omega_T)}\leq 2\delta\theta_n^{s-\alpha -1}.
\label{193}
\end{equation}
Moreover, for all integer $s\in [4,{\alpha}]$ one has
\begin{equation}
\|{\cal B}({U}_n|_{x_1=0},\varphi _n)\|_{H^s(\partial\Omega_T)}\leq \delta\theta_n^{s-\alpha -1}.
\label{194}
\end{equation}
\label{l12}
\end{lemma}

\begin{proof}
One can show that
\begin{equation}
{\cal L}({U}_n,{\Psi}_n)-{f}^a=(S_{\theta_{n-1}}-I){f}^a+(I- S_{\theta_{n-1}}){E}_{n-1}+{e}_{n-1}.\label{195}
\end{equation}
For $s\in [\alpha +1,\tilde{\alpha}-2]$, by using (\ref{72}), we obtain
\[
\|(I- S_{\theta_{n-1}}){f}^a\|_{H^s(\Omega_T)}\leq \theta_n^{s-\alpha -1}(C\|{f}^a\|_{H^{\alpha+1}(\Omega_T)} +\|{f}^a\|_{H^s(\Omega_T)})\leq
C\delta_0(T)\theta_n^{s-\alpha -1},
\]
while for $s\in [3,\alpha +1]$, applying (\ref{73}), we get
\[
\|(I- S_{\theta_{n-1}}){f}^a\|_{H^s(\Omega_T)}\leq C\theta_{n-1}^{s-\alpha -1}\|{f}^a\|_{H^{\alpha +1}(\Omega_T)}\leq
C\delta_0(T)\theta_n^{s-\alpha -1}.
\]
Lemma \ref{l9} and (\ref{73}) imply 
\[
\|(I- S_{\theta_{n-1}}){E}_{n-1}\|_{H^s(\Omega_T)}\leq C\theta_{n-1}^{s-\alpha -2}\|{E}_{n-1}\|_{H^{\alpha+2}(\Omega_T)}\leq
 C\delta^2\theta_{n}^{s-\alpha -1}
\]
for $3\leq s\leq \alpha+2 =\tilde{\alpha}-2$
It follows from (\ref{103}) that
\[
\|{e}_{n-1}\|_{H^s(\Omega_T)}\leq C\delta^2\theta_{n-1}^{L(s)-1}\Delta_{n-1}\leq C\delta^2\theta_{n}^{L(s)-2}\leq C\delta^2\theta_{n}^{s-\alpha -1}.
\]
From the above estimates and decomposition (\ref{195}), by choosing $T>0$ and $\delta>0$ sufficiently small, we obtain (\ref{193}). Similarly, by using the decomposition
\[
{\cal B}({U}_n|_{x_1=0},\varphi_n) =(I-S_{\theta_{n-1}})\widetilde{E}_{n-1}+\tilde{e}_{n-1},
\]
we can prove estimate (\ref{194}).
\end{proof}

As follows from Lemmas \ref{l11} and \ref{l12}, we have proved that $(H_{n-1})$ implies $(H_{n})$, provided that
$\alpha \geq 7$, $\tilde{\alpha}=\alpha +4$, the constant $\theta_0\geq 1$ is large enough, and $T>0$, $\delta >0$ are small enough. Fixing now the constants $\alpha$, $\delta$, and $\theta_0$, we prove $(H_0)$.

\begin{lemma}
If the time $T>0$ is sufficiently small, then $(H_0)$ is true.
\label{l13}
\end{lemma}

\begin{proof}
We recall that $({U}_0,f_0)=0$. Then, by the definition of the approximate solution in Lemma \ref{l2} the state $({U}^a+{U}_0,\varphi^a+\varphi_0)=0$ satisfies already (\ref{6}), (\ref{14}), and (\ref{16'}). That is, it follows from the construction of Proposition \ref{p3} that
$({U}_{1/2},\varphi_{1/2})=0$. Consequently, $(\delta\dot{{U}}_0,\delta \varphi_0)$ solves the linear problem (\ref{26})--(\ref{28}) with the coefficients $(\widehat{{U}},\hat{\varphi})=({U}^a, \varphi^a)$ and the source terms
${f}=S_{\theta_0}{f}^a$ and ${g}=0$. Thanks to (\ref{66}) the assumption (\ref{37}) is satisfied
(recall that $K_0=2C_*$). Applying (\ref{38}), we get the estimate
\[
\|\delta\dot{{U}}_0\|_{H^s(\Omega_T)}+\|\delta \varphi_0\|_{H^s(\partial\Omega_T)}\leq C\|S_{\theta_0}{f}^a\|_{H^{s+1}(\Omega_T)}.
\]
Together with (\ref{69}) and formula (\ref{76}) this estimate yields
\[
\|\delta{{U}}_0\|_{H^s(\Omega_T)}+\|\delta \varphi_0\|_{H^s(\partial\Omega_T)}\leq C\|S_{\theta_0}{f}^a\|_{H^{s+1}(\Omega_T)}
\leq 
C\theta_0^{(s-\alpha)_+}\delta_0(T)\leq \delta\theta_0^{s-\alpha -1}\Delta_0
\]
for all integer $s\in [3,\tilde{\alpha}]$, provided that $T$ is sufficiently small. Likewise, points $b)$ and $c)$ of $(H_0)$ can be shown to be satisfied for a sufficiently short time $T>0$. \hspace{0.4cm}{\scriptsize $\Box$}
\end{proof}

\paragraph{The proof of Theorem \ref{t2}.}
We consider initial data $({U}_0,\varphi_0)\in H^{m+7}(\mathbb{R}^3_+)\times H^{m+7}(\mathbb{R}^2)$ satisfying all the assumptions of Theorem \ref{t2}. In particular, they satisfy the compatibility conditions up to order $\mu=m+7$ (see Definition \ref{d1}). Then, thanks to Lemmas \ref{l1} and \ref{l2} we can construct an approximate solution $({U}^{a},\varphi^a)\in H^{m +8}(\Omega_T)\times H^{m +8}(\partial\Omega_T)$ that satisfies (\ref{66}). As follows from Lemmas \ref{l11}--\ref{l13}, $(H_n)$ holds for all integer $n\geq 0$, provided that $\alpha \geq 7$, $\tilde{\alpha}=\alpha +4$, the constant $\theta_0\geq 1$ is large enough, and the time $T>0$ and the constant $\delta >0$ are small enough.  In particular, ($H_n$) implies
\[
\sum_{n=0}^{\infty}\left\{ \|\delta{U}_n\|_{H^m(\Omega_T)} +\|\delta \varphi_n\|_{H^m(\partial\Omega_T)}\right\} \leq \infty.
\]
Hence, the sequence $({U}_n,\varphi_n)$ converges in $H^{m}(\Omega_T)\times H^{m}(\partial\Omega_T)$ to some
limit $({U} ,\varphi )$. Recall that $m=\alpha -1 \geq 6$. Passing to the limit in (\ref{193}) and (\ref{194}) with
$s=m$, we obtain (\ref{69})--(\ref{71}). Consequently, ${U} :={U} +{U}^a$, $\varphi := \varphi +\varphi^a$ is a solution of problem \eqref{newL}, \eqref{14}, \eqref{15}. This completes the proof of Theorem \ref{t2}.

\section{Free boundary problem in relativistic gas dynamics: special and general relativity}
\label{s6}

Let us first write down a suitable symmetric form of the relativistic Euler equations. First of all, we note that for the set of covariant laws \eqref{9} we have the supplementary covariant law
\begin{equation}
\nabla_{\alpha} (\rho S u^{\alpha})=0
\label{6.1}
\end{equation}
that arises as a consequence of \eqref{9} and the first principle of thermodynamics. In the setting of special relativity \eqref{6.1} becomes the entropy conservation law 
\begin{equation}
\partial_t (\rho \Gamma S) +{\rm div}\,(\rho Su) =0.
\label{6.2}
\end{equation}
In principle, taking into account \eqref{6.2} and using Godunov's symmetrization method, we can rewrite system \eqref{SR1}--\eqref{SR3} for the unknown $U=(p,u,S)$ as a symmetric system for a new (canonical) unknown 
${\cal Q}$ and then return to the original unknown $U$ keeping the symmetry property:
\begin{equation}
A^0(U)\partial_tU+A^j(U)\partial_jU +Q(U)=0,
\label{6.3}
\end{equation}
where $A^{\alpha}=(A^{\alpha})^{\sf T}$, $\partial_j=\partial /\partial x^j$, and $Q(U)= -(0,-\rho{\cal G},0)$. This procedure is described in \cite{BThand} where the symmetric matrices $A^{\alpha}$ were written for the special case $u^2=u^3=0$. Such a procedure is absolutely algorithmic and always works, but it is however connected with very long calculations. Therefore,
here we prefer to symmetrize the conservation laws \eqref{SR1}--\eqref{SR3} by rewriting them in a suitable nonconservative form.

Equations \eqref{SR1} and \eqref{6.2} imply
\begin{equation}
\frac{{\rm d} S}{{\rm d} t} =0,
\label{6.4}
\end{equation}
where ${\rm d} /{\rm d} t =\partial_t + (v,\nabla )$ is the material derivative as for the non-relativistic case \eqref{4}. Using \eqref{6.4}, we first rewrite \eqref{SR1} in a nonconservative form. Combining then \eqref{SR2} and
\eqref{SR3} and employing again \eqref{6.4}, we finally get the relativistic counterpart of system \eqref{4}:
\begin{equation}
\begin{array}{l}
{\displaystyle 
\frac{\Gamma}{\rho c^2}\,\frac{{\rm d} p}{{\rm d} t} +(v,\partial_tu) +{\rm div}\, u =0,}\\[12pt]
{\displaystyle 
(\rho h\Gamma )\left ( \frac{{\rm d} u}{{\rm d} t} -v \left( v,\frac{{\rm d} u}{{\rm d} t} \right)\right)
+(\partial_t p)v +\nabla p =\rho{\cal G},}\\[12pt]
{\displaystyle 
\frac{{\rm d} S}{{\rm d} t} =0,}
\end{array}\label{6.5}
\end{equation}
where  $c=\left(p_\rho(\rho ,S)\right)^{1/2}$. System \eqref{6.5} being written in the 
quasilinear form \eqref{6.3} is already symmetric with 
\begin{equation}
A^0= \left( \begin{array}{ccc}
{\displaystyle \frac{\Gamma}{\rho c^2} } & v^{\sf T} & 0 \\  v & \rho h\Gamma \mathscr{B} & 0 \\
0 & 0 & 1 \end{array}
\right),\qquad  A^{j}= \left( \begin{array}{ccc}
{\displaystyle \frac{u^j}{\rho c^2} } & e_j^{\sf T} & 0 \\  e_j & \rho hu^j  \mathscr{B} & 0 \\
0 & 0 & v^j \end{array}
\right),\label{6.6}
\end{equation}
where 
$\mathscr{B}=(b_{ij}),\quad b_{ij}= \delta_{ij}-v^iv^j,\quad e_j= (\delta_{1j},\delta_{2j},\delta_{3j})$,
and $a^{\sf T}$ is the vector-row for a corresponding vector-column $a$ (recall also that $u^j=\Gamma v^j$).
The matrix $A_0>0$ provided that inequalities \eqref{6} are satisfied together with the {\it relativistic causality} condition 
\begin{equation}
0<c_s^2<1,
\label{6.7}
\end{equation}
where $c_s$ is the relativistic speed of sound, $c_s^2=c^2/h$. Of course, \eqref{6.7} will be an additional restriction on the initial data in a counterpart of Theorem \ref{t2}.

Now, for system \eqref{6.3}, \eqref{6.6} in the domain \eqref{10} endowed with the boundary conditions \eqref{11} we can 
literally repeat arguments of Sections \ref{s3}--\ref{s5}. The only important point is that the boundary matrix ${\cal A}_1$ on the boundary $x_1=0$ for system \eqref{29} written now for matrices \eqref{6.6} and $V=(\dot{p},\dot{u}_n,\dot{u}^2,\dot{u}^3,\dot{S})$ coincides with the matrix ${\cal A}_1|_{x_1=0}$ in \eqref{30}, where
\begin{equation}
\dot{u}_n:=\widehat{\Gamma}\dot{v}_n,\quad \dot{v}_n=\dot{v}^1-\dot{v}^2\partial_2\widehat{\Psi}-\dot{v}^3\partial_3\widehat{\Psi},\quad
\widehat{\Gamma}=(1+|\hat{u}|^2)^{1/2},\quad \hat{v}=\hat{u}/\widehat{\Gamma},\label{.1}
\end{equation}
$\dot{U}=(\dot{p},\dot{u},\dot{S})$ is the ``good unknown'',  $\widehat{U}=(\hat{p},\hat{u},\widehat{S})$ is the basic state, and $\dot{v}=(\dot{v}^1,\dot{v}^2,\dot{v}^3)$ is defined from the formula
\begin{equation}
\dot{u}=\widehat{\Gamma}\dot{v}+\widehat{\Gamma}\hat{u}(\hat{u},\dot{v})
\label{.2}
\end{equation}
suggested by the relation between the perturbations $\delta u$ and $\delta v$. 

Indeed, we easily compute:
\[
\widetilde{A}_1(\widehat{U},\widehat{\Psi})=\frac{1}{\partial_1\widehat{\Phi}}
\left(
\begin{array}{ccccc} 
{\displaystyle \frac{\widehat{\Gamma}\hat{\mathfrak{f}}}{\hat{\rho} \hat{c}^2}}& \hat{a}^{\sf T}&0\\[6pt]
 \hat{a} & \hat{\rho} \hat{h}\widehat{\Gamma} \hat{\mathfrak{f}}\widehat{\mathscr{B}}  &0\\
0&0&\hat{\mathfrak{f}}
\end{array}\right),
\]
where
\[
\hat{a}= (1-\hat{v}_1\partial_t\widehat{\Psi}, -\partial_2\widehat{\Psi} -\hat{v}_2\partial_t\widehat{\Psi}, -\partial_3\widehat{\Psi}-\hat{v}_3\partial_t\widehat{\Psi}),\quad \hat{\mathfrak{f}}=\hat{v}^1-\hat{v}^2\partial_2\widehat{\Psi}-\hat{v}^3\partial_3\widehat{\Psi}
-\partial_t\widehat{\Psi},
\]
and $\widehat{\mathscr{B}}$ is the matrix $\mathscr{B}$ calculated for the basic state. Taking into account \eqref{20},
\eqref{.1}, and \eqref{.2}, we have $\hat{\mathfrak{f}}|_{x_1=0}=0$ and
\[
(\partial_1\widehat{\Phi})\,( \widetilde{A}_1(\widehat{U},\widehat{\Psi})\dot{U},\dot{U} )|_{x_1=0}=2\left.\dot{p}|_{x_1=0}\left(
\dot{u}_1-\dot{u}^2\partial_2\widehat{\Psi}-\dot{u}^3\partial_3\widehat{\Psi}-(\hat{v},\dot{u})\partial_t\hat{\varphi}\right)\right|_{x_1=0}
\]
\[
=2\left.\dot{p}|_{x_1=0}\left( \widehat{\Gamma}\dot{v}_n+(\hat{u},\dot{v})\partial_t\hat{\varphi}(\widehat{\Gamma}^2-1-|\hat{u}|^2)\right)\right|_{x_1=0}=2(\dot{p}\dot{u}_n)|_{x_1=0}=({\cal A}_{(1)}V|_{x_1=0},V|_{x_1=0})
\]
(the matrix ${\cal A}_{(1)}$ was defined in \eqref{30}).
Then
\[
( \widetilde{A}_1(\widehat{U},\widehat{\Psi})\dot{U},\dot{U} )|_{x_1=0}=
( \widetilde{A}_1(\widehat{U},\widehat{\Psi})JV,JV )|_{x_1=0}=( J^{\sf T}\widetilde{A}_1(\widehat{U},\widehat{\Psi})JV,V )|_{x_1=0}=
({\cal A}_1V,V)|_{x_1=0},
\]
where the matrix ${\cal A}_1|_{x_1=0}$ is the same as in Section \ref{s3} and the transition matrix $J$ can be easily written down.
Thus, we obtain the local-in-time existence (and uniqueness) theorem for the relativistic version of problem (\ref{13})--(\ref{15}) (in the framework of special relativity) in the form of Theorem \ref{t2}. Clearly, we should also
supplement conditions \eqref{6} with \eqref{6.7} while writing assumptions on the initial data. It means that the initial data should satisfy 
\[
\inf_{x\in\mathbb{R}^3_+}\left\{ \rho (p_0,S_0), \rho_p (p_0,S_0), c_s^2 (p_0,S_0), 1- c_s^2 (p_0,S_0)\right\}  >0,
\]
where
\[
c_s^2 (p_0,S_0)= \frac{1}{\rho_p (p_0,S_0)h(p_0,S_0)},\qquad h(p_0,S_0)=1+e(\rho (p_0,S_0),S_0) +\frac{p_0}{\rho (p_0,S_0)}.
\]

Let us now briefly discuss the case of general relativity. The metric $g$ appearing in the relativistic Euler equations \eqref{GR} should satisfy the Einstein equations $G_{\alpha\beta}= \kappa T_{\alpha\beta}$. As in \cite{FR}, following Rendall \cite{Rend} and introducing 
\[
g_{\alpha\beta\gamma}:=\partial_{\gamma}g_{\alpha\beta},
\]
we write the Einstein equations in {\it harmonic coordinates} as
\begin{equation}
\begin{array}{l}
-g^{00}\partial_t g_{\alpha\beta 0}-2g^{0i}\partial_i g_{\alpha\beta 0}-g^{ij}\partial_i g_{\alpha\beta j}+
2H_{\alpha\beta}(g_{\gamma\delta},g_{\gamma\delta\sigma} )=\kappa (2T_{\alpha\beta}-T^{\gamma}_{\gamma}g_{\alpha\beta}),\\
g^{ij}\partial_tg_{\alpha\beta i}-g^{ij}\partial_i g_{\alpha\beta i}=0,\\ \partial_tg_{\alpha\beta}-g_{\alpha\beta 0}=0.
\end{array}\label{6.8}
\end{equation}
System \eqref{6.8} written in the compact form
\begin{equation}
B^0(W) \partial_tW+B^j(W)\partial_jW + Q(W,U)=0\label{6.9}
\end{equation}
is symmetric for the vector $W$ whose components are $g_{\alpha\beta}$ and $g_{\alpha\beta\gamma}$. Recall that $U=(p,u,S)$. The symmetric system \eqref{6.9} is hyperbolic if $g^{00}<0$ and $(g^{ij})>0$.

Regarding the relativistic Euler equations \eqref{GR}, it is enough to symmetrize them for a fixed constant metric $g$.
This was done by Rendall \cite{Rend} for isentropic fluids. In the general case we can however just repeat arguments from \cite{Rend} by taking into account the entropy law \eqref{6.2} which has form \eqref{6.4} for constant metrics.
Roughly speaking, the calculations in \cite{Rend} are just a ``tensor'' variant of our simple calculations towards obtaining the nonconservative form \eqref{6.5}. With reference to \cite{Rend}, we write equations \eqref{GR} for a fixed constant metric $g$ in the symmetric form \eqref{6.3}, \eqref{6.6} with
\[
\mathscr{B}=(b_{ij}),\quad b_{ij}=g_{ij}+g_{0i}v^j+g_{0j}v^i+g_{00}v^iv^j=g_{ij}+g_{0i}\frac{u^j}{u^0}+g_{0j}\frac{u^i}{u^0}+g_{00}
\frac{u^iu^j}{(u^0)^2}.
\]
For a non-fixed metric $g$ the balance laws \eqref{GR} are written as the symmetric system 
\begin{equation}
A^0(U)\partial_tU+A^j(U)\partial_jU +B(U,W)=0.
\label{6.10}
\end{equation}
It is worth noting that for system \eqref{6.10} for any fixed (and not necessarily constant) metric we can prove a counterpart of Theorem \ref{t2} under suitable assumptions on $W$.

Now we consider the free boundary problem for the symmetric hyperbolic system \eqref{6.10}, \eqref{6.9} with the boundary conditions \eqref{11}. However, in the setting of general relativity it is actually an interface problem because we should consider system \eqref{6.9} for the metric variables not only in the domain $\Omega (t)$ but also in the vacuum region $\mathbb{R}^3\backslash \Omega(t)= \{ x^1 <\varphi (t,x^2,x^3)\}$. As was shown in \cite{Isr}, the jump conditions on an interface $\Sigma (t)$ written  for the Einstein tensor are satisfied if the metric $g$ is smooth
on this interface, i.e,
\begin{equation}
[W]=W^+-W^-=0 \qquad \mbox{on}\quad \Sigma (t).
\label{6.11}
\end{equation}
In our case $W^+$ and $W^-$ are the metric variables in the fluid domain $\Omega (t)$ and the vacuum region $\mathbb{R}^3\backslash \Omega(t)$ respectively. Constraints on the initial data under which condition \eqref{6.11} is not only sufficient but also necessary for the fulfillment of the jump conditions for the Einstein tensor are discussed in \cite{FR} and connected with the notion of so-called {\it natural coordinates} \cite{Isr}. That is, as for shock waves in general relativity studied in \cite{FR}, we will treat our problem in {\it harmonic natural coordinates}.

Thus, we have the symmetric hyperbolic systems
\begin{align}
& A^0(U)\partial_tU+A^j(U)\partial_jU +B(U,W^+)=0\quad \mbox{in}\ \Omega (t),\label{6.12}\\
& B^0(W^+) \partial_tW^++B^j(W^+)\partial_jW^+ + Q(W^+,U)=0\quad \mbox{in}\ \Omega (t),\label{6.13}\\
& B^0(W^-) \partial_tW^-+B^j(W^+)\partial_jW^- + Q(W^-,0)=0\quad \mbox{in}\ \mathbb{R}^3\backslash\Omega (t)\label{6.14}
\end{align}
endowed with the boundary conditions \eqref{11} and \eqref{6.11} on a time-like hypersurface $\Sigma (t)=\{x^1=\varphi (t,x^2,x^3)\}$. Here \eqref{6.14} is the symmetric form of the vacuum Einstein equations.
We reduce problem \eqref{6.12}--\eqref{6.14}, \eqref{11}, \eqref{6.11} to the fixed domain $\mathbb{R}^3_+$ by straightening the free surface $\Sigma$:
\[
\widetilde{U}(t,x ):= U (t,\Phi^+ (t, x),x'),\quad \widetilde{W}^{\pm}:=W^{\pm} (t,\Phi^{\pm} (t, x),x')
\]
\[
\Phi^{\pm}(t,x ):= \pm x^1+\Psi^{\pm}(t,x ),\quad \Psi^{\pm}(t,x ):= \chi (\pm x_1)\varphi (t,x'),\quad x'=(x^2,x^2)
\]
(the cut-off function $\chi (x_1)$ was described in the beginning of Section \ref{s3}). 

Regarding further arguments towards the proof of the local-in-time existence theorem for the reduced problem in the domain $\mathbb{R}^3_+$, we give here only a {\it rough scheme} or even an idea of this proof and postpone detailed arguments to a future work. The main idea is the following. The existence of solutions of problem \eqref{6.12}, \eqref{11} reduced to the fixed domain $\mathbb{R}^3_+$ is proved by Nash-Moser iterations for any fixed metric $g$. The boundary conditions 
\eqref{6.11} are linear and, therefore, we do not need introduce source terms for them in the linearized problem. 
Moreover, for the linearized problem these boundary conditions are {\it dissipative}. Though, they are not strictly dissipative, but the crucial point is that they are homogeneous. Hence, we can prove the existence of solutions to the reduced problem for \eqref{6.13}, \eqref{6.14}, \eqref{6.11} in $[0,T]\times\mathbb{R}^3_+$ by the classical fixed-point argument for any fixed fluid unknown $U$. Then, the existence of solutions to the whole problem \eqref{6.12}--\eqref{6.14}, \eqref{11}, \eqref{6.11}
reduced to the fixed domain $\mathbb{R}^3_+$ is proved by Nash-Moser iterations for the ``fluid'' part of the problem
whereas at each Nash-Moser iteration step the metric $g$ is found as a solution of the problem whose linear version has maximally dissipative boundary conditions. More presicely, at each $(n+1)$th iteration step before solving the linear problem for $\delta\dot{U}_n$  with $W^+=W^+_n$ we find $W^{\pm}_n$ as a unique solution of the corresponding problem for $W^{\pm}$ with $U=U_n$ and $\varphi=\varphi_n$ taken from the $n$th iteration step. 
At last, we note that the constraints \cite{FR} on the initial data connected with the introduction of natural coordinates  are not needed to be satisfied at each Nash-Moser iteration step and we may therefore not care about them.

\paragraph{Acknowledgements}
The principal part of this work was done during the short stay of the author at the the Department of Mathematics 
and Statistics of the University of Konstanz. The author gratefully thanks Heinrich Freist\"uhler
for his kind hospitality and many helpful discussions during this visit.

\end{document}